\documentclass[smallextended]{svjour3}

\usepackage{amsmath}
\usepackage{amssymb}

\usepackage{enumitem} 

\usepackage{float}
\usepackage{graphicx}
\usepackage{subfig}

\usepackage{xspace}

\usepackage{algpseudocode}
\usepackage{algorithm}

\usepackage{etoolbox} 

\usepackage{hyperref}

\makeatletter
\patchcmd{\ALG@step}{\addtocounter{ALG@line}{1}}{\refstepcounter{ALG@line}}{}{}
\newcommand{\ALG@lineautorefname}{line}
\makeatother

\newcommand{\R}{\mathbb{R}}

\newcommand{\C}{\mathbb{C}}

\DeclareMathOperator{\diag}{diag}
\DeclareMathOperator{\trace}{trace}

\newcommand*\conj[1]{\overline{#1}}

\algdef{SE}[SUBALG]{Indent}{EndIndent}{}{\algorithmicend\ }\algtext*{Indent}
\algtext*{EndIndent}
\algnewcommand{\algorithmicand}{\textbf{ and }}
\algnewcommand{\algorithmicor}{\textbf{ or }}
\algnewcommand{\OR}{\algorithmicor}
\algnewcommand{\AND}{\algorithmicand}
\algnewcommand{\algorithmicinput}{\textbf{Input: }}
\algnewcommand{\INPUT}{\Statex \algorithmicinput}
\algnewcommand{\algorithmicoutput}{\textbf{Output: }}
\algnewcommand{\OUTPUT}{\Statex \algorithmicoutput}
\algnewcommand{\algorithmiccalculation}{\textbf{Calculation: }}
\algnewcommand{\CALCULATION}{\item[\algorithmiccalculation]}
\algrenewcommand{\algorithmiccomment}[1]{
    \hspace*{\fill}
    // #1 //
    \hspace*{\fill}
}
\let\OldReturn\Return
\renewcommand{\Return}{\State\OldReturn}

\begin{document}

\title{Approximation of Hermitian Matrices by\\Positive Semidefinite Matrices using\\Modified Cholesky Decompositions}
\author{Joscha Reimer}
\institute{
Kiel University,\\
Department of Computer Science,\\
Algorithmic Optimal Control - CO$_2$ Uptake of the Ocean,\\
24098 Kiel, Germany\\
\email{joscha.reimer@email.uni-kiel.de}}
\date{May 4, 2019}
\titlerunning{Approximation of Hermitian Matrices by Positive Semidefinite Matrices}
\authorrunning{J. Reimer}
\journalname{Numerical Algorithms}
\maketitle


\begin{abstract}
    A new algorithm to approximate Hermitian matrices by positive semidefinite Hermitian matrices based on modified Cholesky decompositions is presented. In contrast to existing algorithms, this algorithm allows to specify bounds on the diagonal values of the approximation.
    
    It has no significant runtime and memory overhead compared to the computation of a classical Cholesky decomposition. Hence it is suitable for large matrices as well as sparse matrices since it preserves the sparsity pattern of the original matrix. 
        
    The algorithm tries to minimize the approximation error in the Frobenius norm as well as the condition number of the approximation. Since these two objectives often contradict each other, it is possible to weight these two objectives by parameters of the algorithm. In numerical experiments, the algorithm outperforms existing algorithms regarding these two objectives. 
        
    A Cholesky decomposition of the approximation is calculated as a byproduct. This is useful, for example, if a corresponding linear equation should be solved.
    
    A fully documented and extensively tested implementation is available. Numerical optimization and statistics are two fields of application in which the algorithm can be of particular interest.\\
    ~\\

\end{abstract}
\newpage

\section{Introduction} \label{sec: introduction}

Algorithms for approximating Hermitian matrices by positive semidefinite Hermitian matrices are useful in several areas. In stochastics they are needed to transform nonpositive semidefinite estimations of covariance and correlation matrices to valid estimations \cite{Rousseeuw1993,Qi2010,Higham2002a,Higham2016b}. In optimization they are needed to deal with nonpositive definite Hessian matrices in Newton type methods \cite{Gill1981,Nocedal2006,Chong2013}.

The existing algorithms have different disadvantages, which will be outlined below. A new algorithm without these disadvantages is presented in Section \ref{sec:LDL_approximation} where it is also examined in detail. An implementation is introduced in Section \ref{sec:implementation} together with numerical experiments and corresponding results. Conclusions are drawn in Section \ref{sec: summary}.

\subsection{Objectives of approximation algorithms} \label{subsec: introduction: objectives}
 
In order to evaluate the existing algorithms, objectives of an ideal approximation algorithm are established. For this, let $A \in \C^{n \times n}$ be an Hermitian matrix and $B \in \C^{n \times n}$ its approximation. The first three objectives are the following:

\begin{enumerate}[start=1,label={(\textbf{O\arabic*})},leftmargin=*]
    \item $B$ is positive semidefinite.
    \label{objective: positive semidefinite}
\end{enumerate}

\begin{enumerate}[start=2,label={(\textbf{O\arabic*})},leftmargin=*]
    \item The approximation error $\| B - A \|$ is small.
    \label{objective: small error}
\end{enumerate}

\begin{enumerate}[start=3,label={(\textbf{O\arabic*})},leftmargin=*]
    \item The condition number $\kappa(B) = \|B\| \, \|B^{-1}\|$ is small.
    \label{objective: well-conditioned}
\end{enumerate}

In addition to the approximation error, the condition number of the approximation is usually important as well, since, for example, often linear equations including the approximation have to be solved. 

The three objectives \ref{objective: positive semidefinite}, \ref{objective: small error} and \ref{objective: well-conditioned} are sometimes contradictory. Hence, an ideal algorithm would allow to prioritize between \ref{objective: small error} and \ref{objective: well-conditioned}. The norm used in \ref{objective: small error} and \ref{objective: well-conditioned} may depend on the actual application. Typical choices are the spectral norm or the Frobenius norm. 

Especially for large matrices, the execution time of the algorithm as well as the needed memory are important. The fastest way to test whether a matrix is positive definite is to try to calculate its Cholesky decomposition \cite{Higham1988}. This needs $\tfrac{1}{3} n^3 + \mathcal{O}(n^2)$ basic operations in the dense real valued case. The approximation algorithm cannot be expected to be faster but at least asymptotically as fast. Thus, the next two objectives are:

\begin{enumerate}[start=4,label={(\textbf{O\arabic*})},leftmargin=*]
    \item The algorithm requires at most $\mathcal{O}(n^2)$ more basic operations than the calculation of a Cholesky decomposition of $B$.
    \label{objective: fast computation}
\end{enumerate}

\begin{enumerate}[start=5,label={(\textbf{O\arabic*})},leftmargin=*]
    \item The algorithm needs to store $\mathcal{O}(n)$ numbers besides $A$ and $B$ and allows to overwrite $A$ with $B$.
    \label{objective: low memory}
\end{enumerate}

If $A$ is a sparse matrix, $B$ should have the same sparsity pattern. This allows an effective overwriting and is essential if the corresponding dense matrix would be to big to store. Thus, the next objective is:

\begin{enumerate}[start=6,label={(\textbf{O\arabic*})},leftmargin=*]
    \item $A_{i j} = 0$ implies $B_{i j} = 0$.
    \label{objective: preserve sparsity}
\end{enumerate}

For correlation matrices it is crucial that $B$ has only ones as diagonal values. This is the reason for the last objective:
\begin{enumerate}[start=7,label={(\textbf{O\arabic*})},leftmargin=*]
    \item The diagonal of $B$ can be predefined.
    \label{objective: diagonal values}
\end{enumerate}

Similar objectives to \ref{objective: positive semidefinite}, \ref{objective: small error}, \ref{objective: well-conditioned} and \ref{objective: fast computation} have been used in \cite{Schnabel1990,Schnabel1999,Cheng1998,Fang2008}. Here, another objective has been established: If $A$ is "sufficiently" positive definite, $B$ should be equal to $A$. This objective is not explicitly listed here and should be covered by \ref{objective: small error}.

\subsection{Existing approximation methods} \label{subsec: introduction: existing algorithms}

An overview of existing methods to approximate Hermitian matrices by positive semidefinite Hermitian matrices is provided next. They are evaluated using the objectives mentioned above.

The minimal approximation error can be achieved by computing an eigendecomposition and replacing negative eigenvalues \cite{Higham1988,Higham1989}. This was done in statistics \cite{Iman1982,Rousseeuw1993} as well as in optimization \cite[Chapter 3.4]{Nocedal2006}, \cite[Chapter 4.4.2.1]{Gill1981}. However, this does not meet \ref{objective: fast computation}, \ref{objective: preserve sparsity} and \ref{objective: diagonal values}.

It is also possible to calculate approximations with minimal approximation error and the restriction that all diagonal values are one \cite{Higham2002a,Borsdorf2010,Borsdorf2010a,Higham2016}. These methods could be extended so that the approximation has arbitrary predefined (nonnegative) diagonal values. Nevertheless, these methods do not meet \ref{objective: fast computation} and \ref{objective: preserve sparsity}.

Another method, especially common in optimization, is to add a predefined positive definite matrix multiplied by a sufficiently large scalar to the original matrix. The predefined matrix is usually the identity matrix or a diagonal matrix. The scalar is usually determined by increasing a value until the resulting approximation can be successfully Cholesky factorized. This method is also used in a modified Newton's method \cite{Goldfeld1966,Chong2013,Nocedal2006} and the Levenberg-Marquardt method \cite{Levenberg1944,Marquardt1963,Chong2013}. However, \ref{objective: fast computation} and \ref{objective: diagonal values} are not met.

A well-known method, in statistics, is a convex combination with a predefined positive definite matrix. In this context it is based on the concept of shrinkage estimator \cite{Stein1956,Devlin1975,Rousseeuw1993}. The positive definite matrix is again usually the identity matrix or a diagonal matrix. Only the convex combination factor has to be determined. This is usually done by examining the underlying statistical problem \cite{Chen2010,Fisher2011,Ikeda2016,Ledoit2003,Ledoit2004,Schaefer2005,Touloumis2015}. However, methods without using any statistical assumptions exist as well \cite{Higham2016b}. None of these meet \ref{objective: fast computation} and \ref{objective: diagonal values}.

Other methods used, especially in optimization, are modified Cholesky algorithms \cite{Gill1974,Gill1981,Schnabel1990,Schnabel1999,More1979,Cheng1998,Fang2008}. These compute a variant of a Cholesky decomposition like a $LDL^T$, a $LBL^T$ or a $LTL^T$ decomposition. Here $L$ is a lower triangular matrix, $D$ is a diagonal matrix, $B$ is a block diagonal matrix with block size smaller of one or two and $T$ is a tridiagonal matrix. During or after the calculation of these decompositions, their factors are modified so that they represent a positive definite matrix. The methods based on $LBL^T$ decomposition \cite{More1979,Cheng1998} do not meet \ref{objective: fast computation}, \ref{objective: preserve sparsity} and \ref{objective: diagonal values}, the ones based on $LTL^T$ decomposition \cite{Fang2008} do not meet \ref{objective: preserve sparsity} and \ref{objective: diagonal values} and the ones based on $LDL^T$ decomposition \cite{Gill1974,Gill1981,Schnabel1990,Schnabel1999,Fang2008} do not meet \ref{objective: diagonal values}.

Hence, none of the existing methods meet all objectives. However, methods that do not meet \ref{objective: diagonal values} can be extended to meet this objective. For that the calculated approximation is multiplied by a suitable chosen diagonal matrix from both sides. This does not affect \ref{objective: positive semidefinite}, \ref{objective: fast computation}, \ref{objective: low memory} and \ref{objective: preserve sparsity}. So the modified Cholesky method based on $LDL^T$ decomposition could meet all objectives if they are extended to meet \ref{objective: diagonal values}.

The new method presented in Section \ref{sec:LDL_approximation} is a modified Cholesky method based on $LDL^T$ decomposition as well. Contrary to the already published methods of this kind, this methods modifies not only the matrix $D$ but also the matrix $L$ during their calculation. In this way, the algorithm meets all objectives. Furthermore it better meets \ref{objective: small error} and \ref{objective: well-conditioned} than the other extended methods based on $LDL^T$ decomposition as shown in Section \ref{sec:implementation} by numerical experiments.

 \section{The approximation algorithm}
\label{sec:LDL_approximation}

The algorithm \nameref{alg: approximated matrix} which approximates Hermitian matrices by positive semidefinite matrices Hermitian is presented and analyzed in this section.

\subsection{The \nameref{alg: approximated matrix} and the \nameref{alg: approximated decomposition} algorithm }

Previous modified Cholesky methods based on $LDL^T$ decomposition \cite{Gill1974,Gill1981,Schnabel1990,Schnabel1999,Fang2008} applied to a symmetric matrix $A$ try to calculate its $LDL^T$ decomposition. While doing so, they increase some of the values in the diagonal matrix $D$. Hence, they result in a decomposition of a positive definite matrix $A + \Delta$, where $\Delta$ is a diagonal matrix with values greater or equal to zero. However, is this way, the approximation $A + \Delta$ cannot have predefined diagonal elements.

The key idea of the new algorithm is to modify the off-diagonal values of $A$ instead or in addition to its diagonal values. In detail, the Hermitian positive definite approximation $B \in \C^{n \times n}$ of an Hermitian $A\in \C^{n \times n}$ is defined as
\begin{equation*}
    B_{i j}
    :=
    \hat{\omega}_{i j} A_{i j}
    \text{ if }
    i \neq j
    \text{ and }
    B_{i i}
    :=
    A_{i i} + \delta_i
\end{equation*}
where $\hat{\omega}_{i j} \in [0, 1]$, $\hat{\omega}_{i j} = \hat{\omega}_{j i}$ and $\delta_i \in \R$ for all $i, j \in \{1, \ldots, n\}$.

If, for example, $\hat{\omega}_{i j} = 0$ and $\delta_i > |A_{ii}|$ for all $i, j \in \{1, \ldots, n\}$, then $B$ is a diagonal matrix with only positive values and thus positive definite. If, on the other hand, $\hat{\omega}_{i j} = 1$ and $\delta_i = 0$ for all $i, j \in \{1, \ldots, n\}$, then $B = A$ and there is no approximation error.

The challenge is now to determine the values $\hat{\omega}_{i j}$ and $\delta_i$ such that the objectives established in Subsection \ref{subsec: introduction: objectives} are met. This is where we use a (complex valued) modified Cholesky method based on $LDL^H$ decomposition. During the calculation of a $LDL^H$ decomposition of $A$, we modify $L$ and $D$ if the matrix represented by the decomposition would become not positive definite, its condition number would become to high or the requirements on the diagonal values would be violated otherwise.

In detail, the off-diagonal values in the $i$-th row of $L$ are multiplied by $\omega_i \in [0, 1]$ and $\delta_i \in \R$ is added to the $i$-th diagonal value of $D$. This $\delta_i$ corresponds to the previously mentioned $\delta_i$ and $\omega$ to $\hat{\omega}$ such that $\hat{\omega}_{i, j} = \omega_{\max\{{i,j}\}}$ for all $i, j \in \{1, \ldots, n\}$. This relationship is discussed in Subsection \ref{subsec: algorithms: representation}. Furthermore symmetric permutation techniques are used to reduce the approximation error, the computational effort and the required memory.

The algorithm \nameref{alg: approximated decomposition}, which computes the permuted modified $LDL^H$ decomposition and the values $\omega$ and $\delta$, is described in detail in Algorithm \ref{alg: approximated decomposition}.
\begin{algorithm}
\normalsize
\caption{\mbox{DECOMPOSITION}} 
\label{alg: approximated decomposition}

\begin{algorithmic}[1]        
    \INPUT
    \begin{itemize}[label={$\cdot$}]
        \item $A \in \C^{n \times n}$ Hermitian, $x \in (\R \cup \{-\infty\})^n$, $y \in (\R \cup \{\infty\})^n$, $l \in \R \cup \{-\infty\}$, $u \in \R \cup \{\infty\}$,  $\epsilon > 0$
        \item with $\max\{x_i, l\} \leq \min\{y_i, u\}$ for all $i \in \{1, \ldots, n\}$
        \item with $|x_i|, |l| \geq \epsilon$ or $|y_i|, |u| \geq \epsilon$ for all $i \in \{1, \ldots, n\}$
    \end{itemize}

    \OUTPUT
    \begin{itemize}[label={$\cdot$}]
        \item $L \in \C^{n \times n}$, $d, \omega, \delta \in \R^n$, $p \in \{1, \ldots, n\}^n$
\end{itemize}

\Function {decomposition}{$A, x, y, l, u, \epsilon$}

        \State $p_i \gets i$ for all $i \in \{1, \ldots, n\}$ \label{line: alg: approximated decomposition: init p}
        \State $\alpha_i \gets 0$ for all $i \in \{1, \ldots, n\}$ \label{line: alg: approximated decomposition: init alpha}
                    
        \For {$i \gets 1, \ldots, n$} \label{line: alg: approximated decomposition: main for loop}
            
            \State select $j \in \{i, \ldots, n\}$ \label{line: alg: approximated decomposition: permutation}
            \State swap $p_i$ with $p_j$ and $L_{i k}$ with $L_{j k}$ for all $k \in \{1, \ldots, i-1\}$ \label{line: alg: approximated decomposition: swap}    
                            
            \State select $d_i \in [l, u]$, $\omega_{p_i} \in [0,1]$ with $|d_i| \notin (0, \epsilon)$, $d_i + \alpha_{p_i} \omega_{p_i}^2 \in [x_{p_i}, y_{p_i}]$\label{line: alg: approximated decomposition: select d and omega}
            
            \State $L_{i j} \gets \omega_{p_i} L_{i j}$ for all $j \in \{1, \ldots, i-1\}$\label{line: alg: approximated decomposition: L modification omega}
            
            \State $\delta_{p_i} \gets d_i + \omega_{p_i}^2 \alpha_{p_i} - A_{p_i p_i}$ \label{line: alg: approximated decomposition: delta definition}
            
            \For {$j \gets i+1, \ldots, n$}
                \If {$d_i \neq 0$} \label{line: alg: approximated decomposition: if d not zero}
                    \State $L_{j i} \gets \left( A_{p_j p_i} - \sum\limits_{k=1}^{i-1} L_{j k} \conj{L}_{i k} d_k \right) (d_i)^{-1}$ \label{line: alg: approximated decomposition: L definition}
                    \State $\alpha_{p_j} \gets \alpha_{p_j} +  |L_{j i}|^2 d_{i}$ \label{line: alg: approximated decomposition: add alpha}
                \Else
                    \State $L_{j i} \gets 0$ \label{line: alg: approximated decomposition: L set zero}
                \EndIf
            \EndFor
        
        \EndFor
        
        \State $L_{i i} \gets 1$ and $L_{i j} \gets 0$ for all $i, j \in \{1, \ldots, n\}$ with $j > i$ \label{line: alg: approximated decomposition: L set diagonal and upper triangle}
        
        \Return $(L, d, p, \omega, \delta)$
    
    \EndFunction
    
\end{algorithmic}
\end{algorithm}

The algorithm \nameref{alg: approximated matrix}, which computes the approximation $B$, is described in detail in Algorithm \ref{alg: approximated matrix}.

\begin{algorithm}[H]
\normalsize
\caption{\mbox{MATRIX}} 
\label{alg: approximated matrix}

\begin{algorithmic}[1]        
    \INPUT
    \begin{itemize}[label={$\cdot$}]
        \item $A \in \C^{n \times n}$ Hermitian, $x \in (\R \cup \{-\infty\})^n$, $y \in (\R \cup \{\infty\})^n$, $l \in \R \cup \{-\infty\}$, $u \in \R \cup \{\infty\}$,  $\epsilon > 0$
        \item with $\max\{x_i, l\} \leq \min\{y_i, u\}$ for all $i \in \{1, \ldots, n\}$
        \item with $|x_i|, |l| \geq \epsilon$ or $|y_i|, |u| \geq \epsilon$ for all $i \in \{1, \ldots, n\}$
    \end{itemize}
    
    \OUTPUT
    \begin{itemize}[label={$\cdot$}]
        \item $B \in \C^{n \times n}$
    \end{itemize}
    
\Function {matrix}{$A, x, y, l, u, \epsilon$} 
        \State $(L, d, p, \omega, \delta) \gets$ \Call{DECOMPOSITION}{$A, x, y, l, u, \epsilon$}
    
        \State $q_{p_i} \gets i$ for all $i \in \{1, \ldots, n\}$ \label{line: alg: approximated matrix: definition q}
    
        \For {$i \gets 1, \ldots, n$}  \label{line: alg: approximated matrix: outer for loop}
            \State $B_{i i} \gets A_{i i} + \delta_i$ \label{line: alg: approximated matrix: set diagonal}
            \For {$j \gets i+1, \ldots, n$} \label{line: alg: approximated matrix: inner for loop}
                \If {$q_i > q_j$}
                    \State $a \gets j$, $b \gets i$ \label{line: alg: approximated matrix: set a and b case 1}
                \Else
                    \State $a \gets i$, $b \gets j$ \label{line: alg: approximated matrix: set a and b case 2}
                \EndIf
            
                \If {$d_{q_a} \neq 0$ \OR $\omega_b = 0$}
                    \State $B_{i j} \gets A_{i j} \omega_b$ \label{line: alg: approximated matrix: set below diagonal with omega}
                \Else
                    \State $B_{i j} \gets \sum\limits_{k = 1}^{q_a - 1} L_{q_i k} d_k \overline{L}_{q_j k}$ \label{line: alg: approximated matrix: set below diagonal without omega}
                \EndIf
            \EndFor
        \EndFor
        
        \State $B_{j i} \gets \overline{B}_{i j}$ for all $i, j \in \{1, \ldots, n\}$ with $j > i$ \label{line: alg: approximated matrix: set obove diagonal}
    
        \Return $B$
    
    \EndFunction
    
\end{algorithmic}
\end{algorithm}

The parameters $l$ and $u$ of the algorithms are lower and upper bounds on the diagonal values of $D$. The positive definiteness of $B$ can be controlled by $l$ as pointed out in Subsection \ref{subsec: algorithms: positive semidefinite}. The parameters $x$ and $y$ are lower and upper bounds on the diagonal values of $B$ as shown in Subsection \ref{subsec: algorithms: diagonal values}. The condition number of $B$ and the approximation error $\|B- A\|$ are influenced by $x, y, l, u$ as demonstrated in Subsection \ref{subsec: algorithms: condition number} and \ref{subsec: algorithms: approximation error}, respectively. Moreover, they allow to prioritize a low approximation error or a low condition number. The numerical stability of the algorithms is controlled by $\epsilon$.

The algorithms can be considered as a whole class of algorithms since there are many possibilities to choose the permutation and $\omega$ and $\delta$ as discussed in Subsection \ref{subsec: algorithms: choice of d and omega} and \ref{subsec: algorithms: permutation}. The algorithm is carefully designed, so that the overhead in computational effort and memory consumption compared to classical Cholesky decomposition algorithms is negligibly if $\omega$ and $\delta$ are chosen in a proper way, as shown in Subsection \ref{subsec: algorithms: complexity}.

For the rest of this section, we use the following notation for the analysis of both algorithms.
\begin{definition}
    Let
    \begin{equation*}
        B := \text{\normalfont \nameref{alg: approximated matrix}}(A, x, y, l, u, \epsilon)
    \end{equation*}
    where $(A, x, y, l, u, \epsilon)$ is some valid input for the algorithm with $A \in \C^{n \times n}$ and
    \begin{equation*}
        (L, d, p, \omega, \delta) := \text{\normalfont \nameref{alg: approximated decomposition}}(A, x, y, l, u, \epsilon).
    \end{equation*}
    Define $D := \diag(d)$ the diagonal matrix with $d$ as the diagonal. Define $P \in \R^{n \times n}$ as the permutation matrix induced by $p$, which is
    \begin{equation*}
        P_{i j}
        :=
        \begin{cases}
        1 \text{ if } j = p_i\\
        0 \text{ else}
        \end{cases}
        \text{ for all }
        i, j \in \{1, \ldots, n\}
        .
    \end{equation*}
\end{definition}

\subsection{Representation of the approximation matrix} \label{subsec: algorithms: representation}

\begin{sloppypar}
In this subsection it is shown that $B = P^T L D L^H P$. This means that \nameref{alg: approximated matrix} calculates the matrix represented by the decomposition calculated by \nameref{alg: approximated decomposition}. This will be crucial for further investigation of \nameref{alg: approximated matrix}.
\end{sloppypar}

First, we prove that $p$ is a permutation vector.

\begin{lemma} \label{lemma: approximated LDLH decomposition: p uniqueness}
    \begin{equation*}
        \{p_i ~|~ i \in \{1, \ldots, n\}\}
        =
        \{1, \ldots, n\}
        .
    \end{equation*}
\end{lemma}
\begin{proof}
    In \nameref{alg: approximated decomposition}, the variable $p$ is initiated at \autoref{line: alg: approximated decomposition: init p} of the algorithm so that $p_i = i$ for all $i \in \{1, \ldots, n\}$. After its initialization, the variable $p$ is only changed in \autoref{line: alg: approximated decomposition: swap}. Here some of its components are swapped in each iteration. Thus
    $
        \{p_i ~|~ i \in \{1, \ldots, n\}\}
        =
        \{1, \ldots, n\}
    $
    at the end of the algorithm.
    \flushright\qed
\end{proof}

Next it is shown how a corresponding inverse permutation vector can be defined.

\begin{lemma} \label{lemma: p q i = i}
    Define
    \begin{equation*}
        q_{p_i} := i
        \text{ for all }
        i \in \{1, \ldots, n\}
        .
    \end{equation*}
    $q$ is well defined and
    \begin{equation*}
        p_{q_i}
        =
        i
        \text{ for all }
        i \in \{1, \ldots, n\}
        .
    \end{equation*}
\end{lemma}

\begin{proof}
    $q$ is well defined due to Lemma \ref{lemma: approximated LDLH decomposition: p uniqueness}. Let $i \in \{1, \ldots, n\}$. Due to Lemma \ref{lemma: approximated LDLH decomposition: p uniqueness}, a $j \in \{1, \ldots, n\}$ exists with $p_j = i$. Furthermore $q_{p_j} = j$ due to the definition of $q$. Thus, $p_{q_i} = p_{q_{p_j}} = p_{j} = i$ follows.
    \flushright\qed
\end{proof}
    
A fast way to calculate $LDL^H$, using only $A$, $\omega$, $\delta$ and $p$, is pointed out in the next lemma.
\begin{lemma} \label{lemma: approximated LDLH decomposition: decomposition equation}
    \begin{equation*}
        (LDL^H)_{ii}
        =
        A_{p_i p_i} + \delta_{p_i}
    \end{equation*}
    and
    \begin{equation*}
        (LDL^H)_{i j}
        =
        A_{p_i p_j} \omega_{p_{\max\{i, j\}}}
        \text{ if }
        d_{\min\{i, j\}} \neq 0
        \text{ or }
        \omega_{p_{\max\{i, j\}}} = 0
    \end{equation*}
    for all $i, j \in \{1, \ldots, n\}$ with $i \neq j$.
\end{lemma}

\begin{proof}
    First some properties of the variable $p$ during the execution of the algorithm are proved. Denote the for loop starting at \autoref{line: alg: approximated decomposition: main for loop} of the algorithm the main for loop. Let $p^{(0)}$ be the value of the variable $p$ directly before the main for loop and $p^{(i)}$ its value directly after its $i$-th iteration for each $i \in \{1, \ldots, n\}$. Its final value is denoted by $p$.
    
    Let $i \in \{1, \ldots, n\}$. The variable $p$ is initiated so that $p^{(0)}_i = i$. After its initialization, the variable $p$ is only changed in \autoref{line: alg: approximated decomposition: swap}. Here the variables $p_i$ and $p_j$ are swapped for some $j \in \{i, \ldots, n\}$ in the $i$-th iteration of the main for loop. Hence
    \begin{equation} \label{equ: decomposition equation: p permutation vector}
        \{p_i^{(j)} ~|~ i \in \{1, \ldots, n\}\}
        =
        \{1, \ldots, n\}\
        \text{ for all }
        j \in \{1, \ldots, n\}
        .
    \end{equation}
    
    Furthermore the variable $p_i$ is not changed anymore after the $i$-th iteration. Thus
    \begin{equation} \label{equ: decomposition equation: p invariance}
        p_i
        =
        p_i^{(j)}
        \text{ for all }
        i, j \in \{1, \ldots, n\}
        \text{ with }
        i \leq j
    \end{equation}
    and hence
    \begin{equation} \label{equ: decomposition equation: p inequality}
        p_i^{(i)} \neq p_j^{(j)}
        \text{ for all }
        i, j \in \{1, \ldots, n\}
        \text{ with }
        i \neq j.
    \end{equation}
    
    Next it is shown that all entries in the variables $d$, $\omega$ and $\delta$ are set once in the algorithm and are never changed after that. Hence, we do not need an index indicating the current iteration for this variables. Let $d$, $\omega$ and $ \delta$ be the final value of the corresponding variables.
    
    The value of $d_i$ is set in the $i$-th iteration of the main for loop at \autoref{line: alg: approximated decomposition: select d and omega} and nowhere else. The values of $\omega_{p_i}$ and $\delta_{p_i}$ are set in the $i$-th iteration of the main for loop at \autoref{line: alg: approximated decomposition: select d and omega} and \autoref{line: alg: approximated decomposition: delta definition} and due to \eqref{equ: decomposition equation: p inequality} nowhere else. Furthermore $\omega_i$ and $\delta_i$ are set due to equation \eqref{equ: decomposition equation: p invariance} and Lemma \ref{lemma: approximated LDLH decomposition: p uniqueness}. Hence, all entries in the variables $d$, $\omega$ and $\delta$ are set once in the algorithm and are never changed after that.
    
    Next properties of the variable $L$ in the algorithm are proved which will lead to the result of this lemma. Denote with $L^{(i)}$ the value of the variable $L$ directly after the $i$-th iteration of the main for loop for all $i \in \{1, \ldots, n\}$. $L$ denotes its final value.
    
    Let $i, j \in \{1, \ldots, n\}$ with $j < i$. The variable $L_{i j}$ is only changed in the $j$-th iteration at \autoref{line: alg: approximated decomposition: L definition} or \autoref{line: alg: approximated decomposition: L set zero}, in the $i$-th iteration at \autoref{line: alg: approximated decomposition: L modification omega} and maybe in the $k$-th iteration at \autoref{line: alg: approximated decomposition: swap} for $k \in \{j+1, \ldots, i\}$. Thus, after the $i$-th iteration it is unchanged which means
    \begin{equation} \label{equ: decomposition equation: L invariance}
    \begin{gathered}
        L_{i j}
        =
        L_{i j}^{(k)}
        \text{ for all }
        i, j, k \in \{1, \ldots, n\}
        \text{ with }
        j < i \leq k
        .
    \end{gathered}
    \end{equation}
    
    In the $i$-th iteration, the variable $L_{i j}$ might only be changed in \autoref{line: alg: approximated decomposition: swap} and \autoref{line: alg: approximated decomposition: L modification omega}. In \autoref{line: alg: approximated decomposition: swap} the variable $L_{i j}$ is only changed if it is swapped with the variable $L_{k j}$ for some $k \in \{i+1, \ldots, n\}$. This is exactly the case if the variable $p_i$ is swapped with the variable $p_k$. This together with \autoref{line: alg: approximated decomposition: L modification omega} and equation \eqref{equ: decomposition equation: p permutation vector} implies 
    \begin{equation*} 
    \begin{gathered}
        L_{i j}^{(i)}
        =
        \omega_{p_i^{(i)}} L_{k j}^{(i-1)}
        \text{ if }
        p_i^{(i)}
        =
        p_k^{(i-1)}
        \\
        \text{ for all }
        i, j, k \in \{1, \ldots, n\}
        \text{ with }
        j < i
        .
    \end{gathered}
    \end{equation*}
    This results with equation \eqref{equ: decomposition equation: p invariance} and \eqref{equ: decomposition equation: L invariance} in
    \begin{equation} \label{equ: decomposition equation: L change iteration i} 
    \begin{gathered}
        L_{i j}
        =
        \omega_{p_i} L_{k j}^{(i-1)}
        \text{ if }
        p_i
        =
        p_k^{(i-1)}
        \\
        \text{ for all }
        i, j, k \in \{1, \ldots, n\}
        \text{ with }
        j < i
        .
    \end{gathered}
    \end{equation}
    
    In the $k$-th iteration for all $k \in \{j+1, \ldots, i-1\}$, the variable $L_{i j}$ might only be changed in \autoref{line: alg: approximated decomposition: swap} due to a swap with the variable $L_{k j}$. This is exactly the case if the variable $p_i$ is swapped with the variable $p_k$. This together with equation \eqref{equ: decomposition equation: p permutation vector} implies 
    \begin{equation} \label{equ: decomposition equation: L change iteration l}
    \begin{gathered}
        L_{i j}^{(l)} = L_{k j}^{(l-1)}
        \text{ if }
        p_i^{(l)} = p_k^{(l-1)}
        \\
        \text{ for all }
        i, j, k, l \in \{1, \ldots, n\}
        \text{ with }
        j < l < i
        .
    \end{gathered}
    \end{equation}
    
    Equation \eqref{equ: decomposition equation: L change iteration i} and \eqref{equ: decomposition equation: L change iteration l} result in
    \begin{equation} \label{equ: decomposition equation: L changes}
    \begin{gathered}
        L_{i j}
        =
        \omega_{p_i} L_{k j}^{(l)}
        \text{ if }
        p_i
        =
        p_k^{(l)}
        \\
        \text{ for all }
        i, j, k, l \in \{1, \ldots, n\}
        \text{ with }
        j \leq l < i
        .
    \end{gathered}
    \end{equation}
    
    Now with this preparatory work, the main statement of this lemma can be proved. $L_{j j} = 1$ and $L_{j k} = 0$ for all $k \in \{j+1, \ldots, n\}$ due to \autoref{line: alg: approximated decomposition: L set diagonal and upper triangle}. This implies
    \begin{equation*}
        (LDL^H)_{i j}
        =
        \sum\limits_{k=1}^{n} L_{i k} \overline{L}_{j k} d_k
        =
        L_{i j} d_j + \sum\limits_{k=1}^{j-1} L_{i k} \overline{L}_{j k} d_k
        . 
    \end{equation*}
    Due to equation \eqref{equ: decomposition equation: p permutation vector}, a $l \in \{1, \ldots, n\}$ exists with $p_i = p_l^{(j)}$. Hence, equation \eqref{equ: decomposition equation: L invariance} and \eqref{equ: decomposition equation: L changes} imply    
    \begin{equation*}
        L_{i j} d_j + \sum\limits_{k=1}^{j-1} L_{i k} \overline{L}_{j k} d_k
        =
        L_{i j} d_j + \sum\limits_{k=1}^{j-1} L_{i k} \overline{L}_{j k}^{(j)} d_k
        =
        \omega_{p_i} \left( L_{i j}^{(j)} d_j + \sum\limits_{k=1}^{j-1} L_{i k}^{(j)} \overline{L}_{j k}^{(j)} d_k \right)
        .
    \end{equation*}
    Thus
    \begin{equation} \label{equ: decomposition equation: L equals omega sum}
        (LDL^H)_{i j}
        =
        \omega_{p_i} \left( L_{i j}^{(j)} d_j + \sum\limits_{k=1}^{j-1} L_{i k}^{(j)} \overline{L}_{j k}^{(j)} d_k \right)
        . 
    \end{equation}
    
    Due to \autoref{line: alg: approximated decomposition: L definition}
    \begin{equation*}
        A_{p_l^{(j)} p_j^{(j)}}
        =
        L_{l j}^{(j)} d_j + \sum\limits_{k=1}^{i-1} L_{l k}^{(j)} \overline{L}_{j k}^{(j)} d_k
        \text{ if }
        d_j \neq 0
        .
    \end{equation*}
    Furthermore $p_i = p_l^{(j)}$ by definition of $l$ and $p_j = p_j^{(j)}$ due to equation \eqref{equ: decomposition equation: p invariance}. This together with the previous two equations implies
    \begin{equation*}
        (LDL^H)_{i j}
        =
        \omega_{p_i} A_{p_i p_j}
        \text{ if }
        d_j \neq 0
        . 
    \end{equation*}
    
    Moreover with equation \eqref{equ: decomposition equation: L equals omega sum} it follows
    \begin{equation*}
        (LDL^H)_{i j}
        =
        \omega_{p_i} A_{p_i p_j}
        \text{ if }
        \omega_{p_i} = 0
        . 
    \end{equation*}
    
    $D$ is a real-valued diagonal matrix and thus Hermitian. Hence, the matrix $LDL^H$ is Hermitian as well. Since $A$ is also Hermitian, 
    \begin{equation} \label{equ: decomposition equation: offdiagonal values j i}
    \begin{gathered}
        (L D L^H)_{j i}
        =
        (\overline{L D L^H})_{i j}
        =
        \overline{\omega_{p_i} A_{p_i p_j}}
        =
        \omega_{p_i} A_{p_j p_i}
        \text{ if }
        d_j \neq 0
        \text{ or }
        \omega_{p_i} = 0
        .
    \end{gathered}
    \end{equation}
    
    The combination of the three previous equations results in
    \begin{equation*}
    \begin{gathered}
        (LDL^H)_{i j}
        =
        A_{p_i p_j} \omega_{p_{\max\{i, j\}}}
        \text{ if }
        \omega_{p_{\max\{i,j\}}} \neq 0
        \text{ or }
        d_{\min\{i,j\}} = 0\\
        \text{~for all~}
        i, j \in \{1, \ldots, n\}
        \text{~with~}
        i \neq j
    \end{gathered}
    \end{equation*}
    which is one part of the statement of this lemma.
    
    Since $L_{i i} = 1$ and $L_{i k} = 0$ for all $k \in \{i+1, \ldots, n\}$ due to \autoref{line: alg: approximated decomposition: L set diagonal and upper triangle},
    \begin{equation} \label{equ: decomposition equation: diagonal values: sum}
        (L D L^H)_{i i}
        =
        \sum\limits_{j=1}^{n} |L_{i j}|^2 d_j
        =
        d_i + \sum\limits_{j=1}^{i-1} |L_{i j}|^2 d_j
        .
    \end{equation}
    
    Define for every $k \in \{0, \ldots, i-1\}$ an $i_k \in \{1, \ldots, n\}$ with $p_i = p_{i_k}^{(k)}$ which exists uniquely due to equation \eqref{equ: decomposition equation: p permutation vector}. Then equation \eqref{equ: decomposition equation: L changes} implies    
    \begin{equation} \label{equ: decomposition equation: diagonal values: sum with omega}
        \sum\limits_{j=1}^{i-1} |L_{i j}|^2 d_j
        =
        \omega_{p_i}^2 \sum\limits_{k=1}^{i-1} |L_{i_{k} k}^{(k)}|^2 d_{k}
        .
    \end{equation}
    
    Denote with $\alpha^{(0)}$ the value of the variable $\alpha$ directly before the main for loop and with $\alpha^{(i)}$ its value directly after its $i$-th iteration for each $i \in \{1, \ldots, n\}$.
    
    Define for every $k \in \{0, \ldots, i-1\}$ an $i_k \in \{k+1, \ldots, n\}$ with $p_i = p_{i_k}^{(k)}$ which exists uniquely due to equation \eqref{equ: decomposition equation: p permutation vector}. Then
    \begin{equation*}
        \alpha_{p_{i}}^{(i)}
        =
        \alpha_{p_{i_{i-1}}^{(i-1)}}^{(i-1)}
    \end{equation*}
    and
    \begin{equation*}
    \begin{gathered}
        \alpha_{p_{i_k}^{(k)}}^{(k)}
        =
        \alpha_{p_{i_k}^{(k)}}^{(k-1)} + |L_{i_{k} k}^{(k)}|^2 d_{k}
        \text{ for all }
        k \in \{1, \ldots, i-1\}
    \end{gathered}       
    \end{equation*}
    due to \autoref{line: alg: approximated decomposition: add alpha}. Furthermore $\alpha_{p_{i_0}}^{(0)} = 0$ due to \autoref{line: alg: approximated decomposition: init alpha}. Hence
    \begin{equation} \label{equ: decomposition equation: diagonal values: alpha}
        \alpha_{p_{i}}^{(i)}
        =
        \sum\limits_{k=1}^{i-1} |L_{i_{k} k}^{(k)}|^2 d_{k}
        .
    \end{equation}
    
    The combination of equation \eqref{equ: decomposition equation: diagonal values: sum}, \eqref{equ: decomposition equation: diagonal values: sum with omega} and \eqref{equ: decomposition equation: diagonal values: alpha} results in
    \begin{equation*}
        (L D L^H)_{i i}
        =
        d_i + \omega_{p_i}^2  \alpha_{p_{i}}^{(i)}
        .
    \end{equation*}
    
    Due to \autoref{line: alg: approximated decomposition: delta definition} and equation \eqref{equ: decomposition equation: p invariance}, $d_i + \omega_{p_i}^2 \alpha_{p_i}^{(i)} = \delta_{p_i} + A_{p_i p_i}$ and thus
    \begin{equation*}
        (L D L^H)_{i i}
        =
        \delta_{p_i} + A_{p_i p_i}
    \end{equation*}
    which is the other part of the statement of this lemma.
    
    \flushright\qed
\end{proof}

The next lemma shows how $B$ can be calculate using only $A$, $\delta$, $\omega$ and $p$.
\begin{lemma} \label{lemma: approximated LDLH decomposition: B equals A modified}
    \begin{equation*}
        B_{i i}
        =
        A_{i i} + \delta_i
    \end{equation*}
    and
    \begin{equation*}
        B_{i j}
        =
        A_{i j} \omega_{b(i,j)}
        \text{ if }
        d_{q_{a(i,j)}} \neq 0
        \text{ or }
        \omega_{b(i,j)} = 0
    \end{equation*}
    where
    \begin{equation*}
    \begin{gathered}
        q_{p_i} := i,~
        a(i, j) := 
        \begin{cases}
        j \text{ if } q_i > q_j\\
        i \text{ else}
        \end{cases},~
        b(i, j) := 
        \begin{cases}
        i \text{ if } q_i> q_j\\
        j \text{ else}
        \end{cases}
    \end{gathered}
    \end{equation*}
    for all $i, j \in \{1, \ldots, n\}$ with $i \neq j$.
\end{lemma}

\begin{proof}
    First of all, $q$ is well defined due to Lemma \ref{lemma: p q i = i}.     Let $i \in \{1, \ldots, n\}$. In \nameref{alg: approximated matrix}, $B_{i i}$ is set only at \autoref{line: alg: approximated matrix: set diagonal} in the $i$-th iteration of the outer for loop at \autoref{line: alg: approximated matrix: outer for loop}. Due to this line $B_{i i} = A_{i i} + \delta_i$ and thus
    \begin{equation*}
        B_{i i}
        =
        A_{i i} + \delta_i
        \text{ for all }
        i \in \{1, \ldots, n\}
    \end{equation*}
        
    Let $j \in \{i+1, \ldots, n\}$. In \nameref{alg: approximated matrix}, the variable $B_{i j}$ is set only in \autoref{line: alg: approximated matrix: set below diagonal with omega} or \autoref{line: alg: approximated matrix: set below diagonal without omega} in the $i$-th iteration of the outer for loop at \autoref{line: alg: approximated matrix: outer for loop} and the $j$-th iteration of the inner for loop at \autoref{line: alg: approximated matrix: inner for loop}. At this iteration the variables $a$ and $b$ have the the value $a(i, j)$ and $b(i, j)$, respectively, due to \autoref{line: alg: approximated matrix: set a and b case 1} and \autoref{line: alg: approximated matrix: set a and b case 2}. Hence due to \autoref{line: alg: approximated matrix: set below diagonal with omega},
    \begin{equation*}
    \begin{gathered}
        B_{i j}
        =
        A_{i j} \omega_{b(i,j)}
        \text{ if }
        d_{q_{a(i,j)}} \neq 0
        \text{ or }
        \omega_{b(i,j)} = 0\\
        \text{ for all }
        i, j \in \{1, \ldots, n\}
        \text{ with }
        i < j
        .
    \end{gathered}
    \end{equation*}
        
    The variable $B_{j i}$ is set only in \autoref{line: alg: approximated matrix: set obove diagonal} so that $B_{j i} = \overline{B}_{i j}$. Hence, the previous equation implies
    \begin{equation*}
    \begin{gathered}
        B_{j i}
        =
        \overline{B}_{i j}
        =
        \overline{A_{i j} \omega_{b(i,j)}}
        =
        \overline{A}_{i j} \omega_{b(i,j)}
        =
        A_{j i} \omega_{b(j, i)}\\
        \text{ if }
        d_{q_{a(j, i)}} \neq 0
        \text{ or }
        \omega_{b(j, i)} = 0
        \text{ for all }
        i, j \in \{1, \ldots, n\}
        \text{ with }
        i < j
        .
    \end{gathered}
    \end{equation*}
    
    \flushright\qed
\end{proof}

Next the main theorem of this subsection emphasizes the connection between \nameref{alg: approximated matrix} and \nameref{alg: approximated decomposition}.

\begin{theorem} \label{theorem: B equals PT L D LH P}
    \begin{equation*}
        B
        =
        P^T L D L^H P
        .
    \end{equation*}
\end{theorem}

\begin{proof}
    Define
    \begin{equation*}
        q_{p_i} := i
        \text{ for all }
        i \in \{1, \ldots, n\}
        .
    \end{equation*}
    Due to Lemma \ref{lemma: p q i = i}, $q$ is well defined and
    \begin{equation*} 
        p_{q_i}
        =
        i
        \text{ for all }
        i \in \{1, \ldots, n\}
        .
    \end{equation*}
    
    Let $i ,j \in \{1, \ldots, n\}$ with $i < j$. Define $a$ and $b$ so that
    \begin{equation*}
        q_a = \min\{q_i, q_j\}
        \text{ and }
        q_b = \max\{q_i, q_j\}
        .
    \end{equation*}
    This is well defined due to Lemma \ref{lemma: approximated LDLH decomposition: p uniqueness}.
    
    Due to \autoref{line: alg: approximated matrix: set below diagonal without omega} of \nameref{alg: approximated matrix} and the definition of the variables $a$ and $b$ in the algorithm,
    \begin{equation*}
    \begin{gathered}
        B_{i j}
        =
        \sum\limits_{k = 1}^{q_a - 1} L_{q_i k} d_k \overline{L}_{q_j k}
        \text{ if }
        d_{q_a} = 0
        \text{ and }
        \omega_b \neq 0
        .
    \end{gathered}
    \end{equation*}
    Since $L$ is a lower triangular matrix and due to the definition of $q_a$, 
    \begin{equation*}
        L_{q_i k} = 0
        \text{ or }
        L_{q_j k} = 0
        \text{ for all }
        k \in \{q_a + 1, \ldots, n\}
        .
    \end{equation*}
    Thus,    
    \begin{equation*}
    \begin{gathered}
        B_{i j}
        =
        \sum\limits_{k = 1}^{n} L_{q_i k} d_k \overline{L}_{q_j k}
        =
        (L D L^H)_{q_i q_j}
        \text{ if }
        d_{q_a} = 0
        \text{ and }
        \omega_b \neq 0
        .
    \end{gathered}
    \end{equation*}
    Furthermore Lemma \ref{lemma: approximated LDLH decomposition: decomposition equation} and \ref{lemma: approximated LDLH decomposition: B equals A modified} and the definition of $q$ imply
    \begin{equation*}
    \begin{gathered}
        B_{i j}
        =
        A_{i j} \omega_b
        =
        A_{p_{q_i} p_{q_j}} \omega_{p_{q_b}}
        =
        (L D L^H)_{q_i q_j}
        \text{ if }
        d_{q_a} \neq 0
        \text{ or }
        \omega_b = 0
        .
    \end{gathered}
    \end{equation*}
    Due to \autoref{line: alg: approximated matrix: set obove diagonal} of \nameref{alg: approximated matrix},
    \begin{equation*}
    \begin{gathered}
        B_{j i}
        =
        \overline{B}_{i j}
        =
        (\overline{L D L^H})_{q_i q_j}
        =
        (L D L^H)_{q_j q_i}
    \end{gathered}
    \end{equation*}
    Lemma \ref{lemma: approximated LDLH decomposition: decomposition equation} and Lemma \ref{lemma: approximated LDLH decomposition: B equals A modified} imply
    \begin{equation*}
        B_{i i}
        =
        A_{i i} + \delta_{i}
        =
        A_{p_{q_i} p_{q_i}} + \delta_{p_{q_i}}
        =
        (L D L^H)_{q_i q_i}
        .
    \end{equation*}
    Thus
    \begin{equation*}
    \begin{gathered}
        B_{i j}
        =
        (L D L^H)_{q_i q_j}
        \text{ for all }
        i, j \in \{1, \ldots, n\}
        .
    \end{gathered}
    \end{equation*}
    The definition of $P$ implies
    \begin{equation*}
        P_{q_i i} = 1
        \text{ and }
        P_{q_j i} = 0
        \text{ for all }
        i, j \in \{1, \ldots, n\}
        \text{ with }
        i \neq j
        .
    \end{equation*}
    Hence,
    \begin{equation*}
    \begin{gathered}
        (L D L^H)_{q_i q_j}
        =
        \sum\limits_{k=0}^n \sum\limits_{j=0}^n
        P_{k i} (L D L^H)_{k l} P_{l j}
        =
        (P^T L D L^H P)_{i j}
        \\
        \text{ for all }
        i, j \in \{1, \ldots, n\}
        .
    \end{gathered}
    \end{equation*}
    
    \flushright\qed
\end{proof}

\subsection{Positive semidefinite approximation} \label{subsec: algorithms: positive semidefinite}

\nameref{alg: approximated matrix} can be forced to calculate positive definite or positive semidefinite matrices using $l > 0$ or $l \geq 0$, respectively as shown in Theorem \ref{theorem: algorithms: positive semidefinite}. Thus, \nameref{alg: approximated matrix} meets objective \ref{objective: positive semidefinite} if $l \geq 0$ is chosen. To prove this theorem, it is first shown that the values of $d$ are bounded below by $l$. For subsequent proofs, it is also shown that the values of $d$ are bounded above by $u$ and $y$.
\begin{lemma} \label{lemma: approximated LDLH decomposition: d bounds}
    \begin{equation*}
        d_i
        \in
        [l, u] \cap \R
        \text{ and }
        |d_i| \notin (0, \epsilon)
    \end{equation*}
    and if $l \geq 0$,
    \begin{equation*}
        d_i \leq y_{p_i}
    \end{equation*}
    for all $i \in \{1, \ldots, n\}$.
\end{lemma}

\begin{proof}
    Let $i \in \{1, \ldots, n\}$. In \nameref{alg: approximated decomposition} the variable $d$ is only changed in \autoref{line: alg: approximated decomposition: select d and omega}. Here $d_i$ is chosen at the $i$-th iteration of the surrounding for loop so that $d_i \in [l, u] \cap \R$ and $|d_i| \notin (0, \epsilon)$. Apart from that, the variable $d_i$ is not set or changed anymore, so
    \begin{equation*}
        d_i
        \in
        [l, u] \cap \R
        \text{ and }
        |d_i| \notin (0, \epsilon)
        \text{ for all }
        i \in \{1, \ldots, n\}
        .
    \end{equation*}
    
    The variable $\alpha$ in \nameref{alg: approximated decomposition} is only changed in \autoref{line: alg: approximated decomposition: init alpha} and \autoref{line: alg: approximated decomposition: add alpha}. Due to this lines and the previous equation,
    \begin{equation*}
        \alpha_i \geq 0
        \text{ if }
        l \geq 0
        .
    \end{equation*}
    
    In \autoref{line: alg: approximated decomposition: select d and omega}, $d_i$ is also chosen so that $d_i + \omega_{p_i}^2 \alpha_{p_i} \leq y_{p_i}$. This implies, together with the previous equation,
    \begin{equation*}
        d_i \leq y_{p_i}
        \text{ if }
        l \geq 0
        \text{ for all }
        i \in \{1, \ldots, n\}
        .
    \end{equation*}
    
    \flushright\qed
\end{proof}

\begin{theorem} \label{theorem: algorithms: positive semidefinite}
    $B$ is positive semidefinite if $l \geq 0$ and positive definite if $l > 0$.
\end{theorem}
\begin{proof}
    Theorem \ref{theorem: B equals PT L D LH P} implies
    \begin{equation*}
        z^H B z
        =
        z^H P^T L D L^H P z
        =
        (L^H P z)^H D (L^H P z)
    \end{equation*}
    for all $z \in \C^n$. Moreover $L$ and $P$ are invertible. Hence, $B$ is positive semidefinite if $D_{i i} = d_i \geq 0$ and positive definite if $D_{i i} = d_i > 0$ for all $i \in \{1, \ldots, n\}$. Thus, Lemma \ref{lemma: approximated LDLH decomposition: d bounds} implies that $B$ is positive semidefinite if $l \geq 0$ and positive definite if $l > 0$.
    
\end{proof}

\subsection{Diagonal values} \label{subsec: algorithms: diagonal values}

\nameref{alg: approximated matrix} allows to define lower and upper bounds for the diagonal values of $B$ using $x$ and $y$ as proved in Theorem \ref{theorem: algorithms: approximated matrix: diagonal values}. This allows to predefined diagonal values of $B$ by setting both bounds to the desired diagonal values. Thus, \nameref{alg: approximated matrix} meets objective \ref{objective: diagonal values} by appropriately selecting the parameters $x$ and $y$.

It should be taken into account that the algorithm requires $x_i \leq u$ and $l \leq y_i$ for all $i \in \{1, \ldots, n\}$. Hence, if positive semidefinite approximations are required, only nonnegative values can be used as predefined diagonal values. However, this is not an actual restriction, since positive semidefinite matrices always have nonnegative diagonal values.

\begin{theorem} \label{theorem: algorithms: approximated matrix: diagonal values}
    \begin{equation*}
        x_i
        \leq
        B_{i i}
        \leq
        y_i
        \text{ for all }
        i \in \{1, \ldots, n\}
        .
    \end{equation*}
\end{theorem}
\begin{proof}
    In the \nameref{alg: approximated matrix}, \nameref{alg: approximated decomposition} is called first to calculate $L, d, p, \omega$ and $\delta$. Let $i \in \{1, \ldots, n\}$. At the $i$-th iteration of the outer for loop in \nameref{alg: approximated decomposition},
    \begin{equation*}
        d_i + \omega_{p_i}^2 \alpha_{p_i}
        \in
        [x_{p_i}, y_{p_i}]
    \end{equation*}
    due to \autoref{line: alg: approximated decomposition: select d and omega} and 
    \begin{equation*}
        \delta_{p_i}
        =
        d_i + \omega_{p_i}^2 \alpha_{p_i} - A_{p_i p_i}
    \end{equation*}
    due to \autoref{line: alg: approximated decomposition: delta definition} and thus also
    \begin{equation*}
        A_{p_i p_i} + \delta_{p_i}
        \in
        [x_{p_i}, y_{p_i}]
        .
    \end{equation*}
    The variables $p_i$ and $\delta_{p_i}$ are not changed anymore after that. Thus
    \begin{equation*}
        A_{p_i p_i} + \delta_{p_i}
        \in
        [x_{p_i}, y_{p_i}]
        \text{ for all }
        i \in \{1, \ldots, n\}
    \end{equation*}
    at the end of the algorithm. Due to Lemma \ref{lemma: approximated LDLH decomposition: p uniqueness}, 
    \begin{equation*}
        \{p_i \in \{1, \ldots, n\}\} = \{1, \ldots, n\}
    \end{equation*}
    and thus
    \begin{equation*}
        A_{i i} + \delta_{i}
        \in
        [x_{i}, y_{i}]
        .
    \end{equation*}
    Lemma \ref{lemma: approximated LDLH decomposition: B equals A modified} states that
    \begin{equation*}
        B_{i i}
        =
        A_{i i} + \delta_i
    \end{equation*}
    and thus
    \begin{equation*}
        x_{i} \leq B_{i i} \leq y_{i}
        .
    \end{equation*}   
    
    \flushright\qed 
\end{proof}

\subsection{Condition number} \label{subsec: algorithms: condition number}

The condition number of $B$ can be controlled by $l, u$ and $y$ as shown in Theorem \ref{theorem: algorithms: condition number}. Hence, \nameref{alg: approximated matrix} meets objective \ref{objective: well-conditioned} with suitable chosen parameters.

\begin{theorem} \label{theorem: algorithms: condition number}
    Let $l > 0$. Then
    \begin{equation*}
    \begin{gathered}
        \kappa_2(L)
        \leq
        2 \left( \frac{a}{l} \right)^\frac{n}{2},~
        \kappa_2(D)
        \leq
        \frac{b}{l}
        \text{ and }
        \kappa_2(B)
        \leq
        4 \frac{a^n b}{l^{n+1}}\\
        \text{with }
        a := \frac{1}{n} \sum\limits_{i=1}^{n} y_i
        \text{ and }
        b := \min\{u, \max_{i=1,\ldots,n} y_i\}
        .
    \end{gathered}
    \end{equation*}
\end{theorem}

\begin{proof}
    $P$ is a permutation matrix and thus $\trace(P B P^T) = \trace(B)$. Furthermore, $P B P^T$ is positive definite because a permutation matrix is invertible and $B$ is positive definite due to Theorem \ref{theorem: algorithms: positive semidefinite}. Moreover, $\kappa_2(P B P^T) = \kappa_2(B)$ because a permutation matrix is also orthogonal. Thus, Theorem \ref{theorem: algorithms: approximated matrix: diagonal values} implies
\begin{equation*}
        \frac{\trace(P B P^T)}{n}
        =
        \frac{\trace(B)}{n}
        \leq
        \frac{1}{n} \sum\limits_{i=1}^{n} y_i
        =
        a
        .
    \end{equation*}
    Theorem \ref{theorem: B equals PT L D LH P} states that
    \begin{equation*}
        P B P^T
        =
        L D L^H
        .
    \end{equation*}
    Lemma \ref{lemma: approximated LDLH decomposition: d bounds} implies
    \begin{equation*}
        l
        \leq
        D_{i i}
        \leq
        \min\{u, y_{p_i}\}
        \leq
        b
        \text{ for all }
        i \in \{1, \ldots, n\}
    \end{equation*}
    since $l \geq 0$. Hence, Theorem \ref{theorem: condition number inequality: positive definite} in the appendix implies
    \begin{equation*}
        \kappa_2(L)
        \leq
        2 \left( \frac{a}{l} \right)^\frac{n}{2},~
        \kappa_2(D)
        \leq
        \frac{b}{l}
        \text{ and }
        \kappa_2(B)
        \leq
        4 \frac{a^n b}{l^{n+1}}
        .
    \end{equation*}
    
    \flushright\qed
\end{proof}

\subsection{Approximation error} \label{subsec: algorithms: approximation error}

The approximation error $\| B - A \|$ can be expressed using $A$, $\delta$, $\omega$ and $p$ as shown in the next theorem where it is also proved that the approximation error is bounded. For that, it is first demonstrated that $\delta$ is bounded.

\begin{lemma} \label{lemma: algorithms: bound delta}
    Let $l \geq 0$. Then
    \begin{equation*}
    \begin{gathered}
        |\delta_i| \leq a + b
        \text{ for all }
        i \in \{1, \ldots, n\}\\
        \text{ with }
        a := \max\limits_{i=1,\ldots,n} y_i
        \text{ and }
        b := \max\limits_{i=1,\ldots,n} |A_{ii}|
        .
    \end{gathered}
    \end{equation*}
\end{lemma}

\begin{proof}
    Let $i \in \{1, \ldots, n\}$. $B$ is positive semidefinite due to Theorem \ref{theorem: algorithms: positive semidefinite} since $l \geq 0$. Hence,
    \begin{equation*}
        0
        \leq
        B_{i i}
        \text{ and }
        B_{i i}
        \leq
        y_i
        \leq
        a
    \end{equation*}
    due to Theorem \ref{theorem: algorithms: approximated matrix: diagonal values}. Furthermore
    \begin{equation*}
        B_{i i}
        =
        A_{i i} + \delta_i
    \end{equation*}
    due to Lemma \ref{lemma: approximated LDLH decomposition: B equals A modified}. Thus
    \begin{equation*}
        | \delta_i |
        =
        |B_{i i} - A_{i i}|
        \leq
        |B_{i i}| + |A_{i i}|
        \leq
        a + b
        .
    \end{equation*}
    
    \flushright\qed
\end{proof}

\begin{theorem} \label{theorem: algorithms: approximation error}
    Let $l > 0$ or otherwise $d_i = 0$ imply $\omega_j = 0$ for all $i, j \in \{1, \ldots, n\}$ with $j \geq i$.
    Define $E := B - A$. Then
    \begin{equation*}
    \begin{aligned}
        \|E\|_2
        &\leq
        \|E\|_1
        =
        \|E\|_\infty
        \\&=
        \max\limits_{i = 1, \ldots, n} \left(
        |\delta_{p_i}| + (1 - \omega_{p_i}) \sum\limits_{j=1}^{i-1} |A_{p_i p_j}| + \sum\limits_{j=i+1}^n (1 - \omega_{p_j}) |A_{p_i p_j}|
        \right)    
        \\&\leq
        a + b + (n - 1) c
    \end{aligned}
    \end{equation*}
    and
    \begin{equation*}
    \begin{aligned}
        \|E\|_F^2
        &=
        \sum\limits_{i=1}^n 
        \left(
        \delta_{p_i}^2
        +
        2 (1 - \omega_{p_i})^2 \sum\limits_{j=1}^{i-1} |A_{p_i p_j}|^2
        \right)
        \\&\leq
        n ((a + b)^2 + (n - 1) c^2)
    \end{aligned}
    \end{equation*}
    with
    \begin{equation*}
        a := \max\limits_{i=1,\ldots,n} y_i,
        b := \max\limits_{i=1,\ldots,n} |A_{ii}|
        \text{ and }
        c := \max\limits_{i,j=1,\ldots,n; i \neq j} |A_{ij}|
        .
    \end{equation*}
\end{theorem}

\begin{proof}        
    Let $i, j \in \{1, \ldots, n\}$. Lemma \ref{lemma: approximated LDLH decomposition: decomposition equation} and Theorem \ref{theorem: B equals PT L D LH P} imply
    \begin{equation*}
    \begin{aligned}
        B_{p_i p_j}
        =
        \begin{cases}
        A_{p_i p_j} \omega_{p_{\max\{i, j\}}} &\text{ if } i \neq j\\
        A_{p_i p_i} + \delta_{p_i} &\text{ else}
        \end{cases}
        .
    \end{aligned}
    \end{equation*}
    Thus, 
    \begin{equation*}
    \begin{aligned}
        E_{p_i p_j}
        =
        \begin{cases}
        (\omega_{p_{\max\{i, j\}}} - 1) A_{p_i p_j}  &\text{ if } i \neq j\\
        \delta_{p_i} &\text{ else}
        \end{cases}
        .
    \end{aligned}
    \end{equation*}
    Furthermore 
    \begin{equation*}
        \{p_i ~|~ i \in \{1, \ldots, n\}\}
        =
        \{1, \ldots, n\}
    \end{equation*}
    due to Lemma \ref{lemma: approximated LDLH decomposition: p uniqueness}.
    Hence, $E$ is Hermitian because $A$ is Hermitian. Thus, the properties of the norms imply
    \begin{equation*}
        \|E\|_2
        \leq
        \|E\|_1
        =
        \|E\|_\infty
        .
    \end{equation*}  
    Moreover
    \begin{equation*}
    \begin{aligned}
        \|E\|_\infty
        &=
        \max\limits_{i = 1, \ldots, n} \sum\limits_{j = 1}^{n} |E_{p_i p_j}|
        =
        \max\limits_{i = 1, \ldots, n} \left(
        |E_{p_i p_i}| + \sum\limits_{j=1}^{i-1} |E_{p_i p_j}| + \sum\limits_{j=i+1}^n |E_{p_i p_j}|
        \right)
        \\&=
        \max\limits_{i = 1, \ldots, n} \left(
        |\delta_{p_i}| + (1 - \omega_{p_i}) \sum\limits_{j=1}^{i-1} |A_{p_i p_j}| + \sum\limits_{j=i+1}^n (1 - \omega_{p_j}) |A_{p_i p_j}|
        \right)    
        \\&\leq
        a + b + (n - 1) c   
    \end{aligned}
    \end{equation*}
    because $|\delta_i| \leq a + b$ and $\omega_i \in [0, 1]$
due to Lemma \ref{lemma: algorithms: bound delta} and line \ref{line: alg: approximated decomposition: select d and omega} in \nameref{alg: approximated decomposition}. Additionally    
    \begin{align*}
        \|E\|_F^2
        &=
        \sum\limits_{i=1}^n
        \left(
        |E_{p_i p_i}|^2
        +
        2 \sum\limits_{j=1}^{i-1} |E_{p_i p_j}|^2
        \right)
        \\&=
        \sum\limits_{i=1}^n
        \left(
        \delta_{p_i}^2
        +
        2 (1- \omega_{p_i})^2
        \sum\limits_{j=1}^{i-1} | A_{p_i p_j}|^2
        \right)
        \\&\leq
        n (a + b)^2 + n (n - 1) c^2
        .
    \end{align*}
    
    \flushright\qed
\end{proof}

\subsection{Choice of $\omega$ and $d$}  \label{subsec: algorithms: choice of d and omega}

The choice of $\omega$ and $d$ in \autoref{line: alg: approximated decomposition: select d and omega} in \nameref{alg: approximated decomposition} is arbitrary apart from that they must be feasible. However, their choice is crucial for the approximation error due to Theorem \ref{theorem: algorithms: approximation error} and line \ref{line: alg: approximated decomposition: delta definition} of \nameref{alg: approximated decomposition}.

Based on this theorem the algorithm \nameref{alg: minimal change}, presented in Algorithm \ref{alg: minimal change}, is derived which chooses $\omega$ and $d$ so that in each iteration the additional approximation error in the Frobenius norm is minimized. This does not guaranteed that the overall approximation error is minimized but still results in a small approximation error as numerical tests in Subsection \ref{subsec: numerical comparison} have shown. Hence, \nameref{alg: approximated matrix} meets objective \ref{objective: small error} when using \nameref{alg: minimal change}. It can be incorporated by replacing \autoref{line: alg: approximated decomposition: select d and omega} in \nameref{alg: approximated decomposition} with the code snippet \nameref{alg: approximated decomposition extension: minimal change} presented in Algorithm \ref{alg: approximated decomposition extension: minimal change}.

\nameref{alg: minimal change} was designed so that its needed number of basic operations and memory is negligible compared to the number of basic operations and memory needed by \nameref{alg: approximated matrix} as discussed in Subsection \ref{subsec: algorithms: complexity}. This makes it possible to meet objectives \ref{objective: fast computation} and \ref{objective: low memory} while using \nameref{alg: minimal change}.

It also ensures that $B = A$ if $A$ already meets the requirements on $B$. In detail, these are $x_i \leq A_{ii} \leq y_i$ and $\max \{l, \epsilon\} \leq D_{ii} \leq u$ for all $i \in \{1, \ldots, n\}$, where $D$ is the diagonal matrix of the $LDL^H$ decomposition of $P A P^T$.

If several pairs $(d, \omega)$ minimize the additional approximation error, the one with the biggest $d$ is chosen in \nameref{alg: minimal change}. This results in absolute smaller values in $L$ which reduces the condition number of $B$, as shown in the proof of Theorem \ref{theorem: condition number inequality: positive definite}. Moreover the numerical stability of the algorithms is increased because a division by $d$ is part of the algorithms.

\begin{algorithm}
\normalsize
\caption{MINIMAL\_CHANGE}
\label{alg: minimal change}

\begin{algorithmic}[1]
    \INPUT
    \begin{itemize}[label={$\cdot$}]
        \item $x \in \R \cup \{-\infty\}$, $y, u \in \R \cup \{\infty\}$, $l, \epsilon, \alpha, \beta, \gamma \in \R$ with
            $l, \alpha, \beta \geq 0$, $\epsilon > 0$,
            $\max\{l, \epsilon, x\} \leq \min\{u, y\}$ and
            $\beta = 0 \Rightarrow \alpha = 0$
    \end{itemize}
    
    \OUTPUT
    \begin{itemize}[label={$\cdot$}]
        \item $d, \omega \in \R$
    \end{itemize}

    \Function {minimal\_change}{$x, y, l, u, \epsilon, \alpha, \beta, \gamma$}
        \If {$\max\{l, \epsilon, x - \alpha\} \leq \gamma - \alpha \leq \min\{u, y - \alpha\}$} \label{line: alg: select omega and d: case inner value: if condition}
            \Return $(\gamma - \alpha, 1)$ \label{line: alg: select omega and d: case inner value: return values}
        \EndIf
                    
        \State $C \gets \emptyset$
        
        \If {$\max\{l, \epsilon, x - \alpha\} \leq \min\{u, y - \alpha\}$}
            \State $C \gets  \{(\min\{\max\{l, \epsilon, x - \alpha, \gamma - \alpha\}, u, y - \alpha\}, 1) \}$ \label{line: alg: select omega and d: add candidate: case omega equals 1}
        \EndIf
        \If {$\alpha \neq 0$}
            \For {$d \in (\{\max\{l, \epsilon\} \} \cap [x -a, \infty)) \cup (\{u\} \cap (- \infty, y]))$} 
                \For {$\omega \in \R \text{ with } 2 \alpha^2 \omega^3 + ( 2 \alpha (d - \gamma) + \beta ) \omega - \beta = 0$}
                    \State $\omega \gets 
                            \min \{
                                 \max \{
                                     \omega,
                                     \sqrt{\tfrac{\max\{x - d, 0\}}{\alpha}}
                                 \},
                                 \sqrt{\tfrac{y - d}{\alpha}},
                                 1
                             \}$
                    \State $C \gets C \cup \{(d, \omega)\}$ \label{line: alg: select omega and d: add candidate: omega inner}
                \EndFor
            \EndFor
        \EndIf
            
        \If {$l = 0$ and $x \leq 0$ and $2 \gamma \leq \epsilon$}
            \State $C \gets C \cup \{(0, 0)\}$ \label{line: alg: select omega and d: add candidate: (0, 0)}
        \EndIf
        
        \Return $(d, \omega) \in C$ with smallest $((d + \omega^2 \alpha  - \gamma)^2 + (\omega - 1)^2 \beta, -d, \omega)$ in lexicographical order \label{line: alg: select omega and d: return minimal candidate}
        
    \EndFunction
\end{algorithmic}
\end{algorithm}

\begin{algorithm}[H]
\normalsize
\caption{CHOOSE\_$d$\_$\omega$} 
\label{alg: approximated decomposition extension: minimal change}

\begin{algorithmic}[1]
    \Indent
    \Indent
    
    \For {$k \gets i, \ldots, n$}
    \If {$i = 1$}
        \State $\beta_{p_k} \gets 0$
    \Else
        \State $\beta_{p_k} \gets \beta_{p_k} + 2 |A_{p_k p_{i- 1}}|^2$
    \EndIf
    \EndFor
    
    \State $(d_{p_i}, \omega_{p_i}) \gets $\Call{minimal\_change}{$x_{p_i}, y_{p_i}, l, u, \epsilon, \alpha_{p_i}, \beta_{p_i}, A_{p_i p_i}$} \label{line: alg: approximated decomposition extension: minimal change: select minimum} 
    
    \EndIndent
    \EndIndent
\end{algorithmic}
\end{algorithm}

The next Theorem stats that \nameref{alg: minimal change} chooses feasible $d$ and $\omega$ which minimize in each iteration the additional approximation error.

\begin{theorem} \label{theorem: optimal omega and d}
    Let
    \begin{equation*}
        d, \omega := \text{\normalfont \nameref{alg: minimal change}}(x, y, l, u, \epsilon, \alpha, \beta, \gamma)
    \end{equation*}
    where $(x, y, l, u, \epsilon, \alpha, \beta, \gamma)$ is some valid input for the algorithm. Let
    \begin{equation*}
        \Phi_* :=
        \{
        (d, \omega)
        ~|~
        d \in [\max\{l, \epsilon\}, u],
        \omega \in [0, 1],
        d + \omega^2 \alpha \in [x, y]
        \}
        ,
    \end{equation*}
    \begin{equation*}
        \Phi_0 :=
        \begin{cases}
            \{(0, 0)\} \text{ if } \max\{l, x\} \leq 0\\
            \emptyset \text{ else}
        \end{cases}
        ,
        \Phi := \Phi_* \cup \Phi_0
    \end{equation*}
    and
    \begin{equation*}
        \Psi :=
        \{
        (d, \omega) \in \Phi
        ~|~
        f(d,\omega)
        =
        \min_{(\hat{d}, \hat{\omega}) \in \Phi} f(\hat{d}, \hat{\omega})
        \}
    \end{equation*}   
    with $f: \R^2 \to \R, (d, \omega) \mapsto (d + \omega^2 \alpha - \gamma)^2 + (\omega - 1)^2 \beta$. Then $(d, \omega) \in \Psi$.
\end{theorem}

\begin{proof}
    $\Phi$ is compact and $f$ is continuous. Thus, $f$ has a minimum on $\Phi$ due to Weierstrass's theorem \cite[Theorem 4.16]{Rudin1976}. Hence, $\Psi \neq \emptyset$ and thus,
    \begin{equation} \label{eq: theorem: optimal omega and d: psi nonempty sets}
        \Psi \cap \Phi_*^\circ \neq \emptyset
        \text{ or }
        \Psi \cap \partial \Phi_* \neq \emptyset
        \text{ or }
        \Psi \cap \Phi_0 \neq \emptyset
    \end{equation}
    where $\Phi_*^\circ$ denotes the interior of $\Phi_*$ and $\partial \Phi_*$ its boundary. Next these three cases are considered.
        
    First consider the case that $\Psi \cap \Phi_*^\circ \neq \emptyset$. Then
    \begin{equation*}
        \nabla f(d, \omega) = 0
        \text{ for all } 
        (d, \omega) \in \Psi \cap \Phi_*^\circ
    \end{equation*}
    due to \cite[Theorem 12.3]{Nocedal2006}. Furthermore
    \begin{equation*}
        \nabla f(d, \omega)
        =
        \left(
        \begin{array}{c}
        2 (d + \omega^2 \alpha - \gamma)\\
        4 \alpha \omega (d + \omega^2 \alpha - \gamma) + 2 \beta (\omega - 1)
        \end{array}
        \right)
    \end{equation*}
    for all $(d, \omega) \in \Phi_*^\circ$. This implies
    \begin{equation*}
        \omega = 1
        \text{ and }
        d = \gamma - \alpha
        \text{ for all }
        (d, \omega) \in \Psi \cap \Phi_*^\circ
        \text{ if }
        \beta \neq 0
        .
    \end{equation*}
    If $\beta = 0$, the algorithm requires $\alpha = 0$, which implies
    \begin{equation*}
        (\gamma - \alpha, 1) \in \Psi
        \text{ if }
        \Psi \cap \Phi_*^\circ \neq \emptyset
        \text{ and }
        \beta = 0
        .
    \end{equation*}
    Hence,
    \begin{equation*}
        (\gamma - \alpha, 1) \in \Psi
        \text{ if }
        \Psi \cap \Phi_*^\circ \neq \emptyset
        .
    \end{equation*}
    Thus, $\Psi \cap \Phi_*^\circ \neq \emptyset$ implies $(\gamma - \alpha, 1) \in \Phi_*$. Hence, $(\gamma - \alpha, 1)$ is returned by the algorithm in \autoref{line: alg: select omega and d: case inner value: return values} if $\Psi \cap \Phi_*^\circ \neq \emptyset$.
    
    If $\Psi \cap \Phi_*^\circ = \emptyset$, the algorithm constructs a candidate set $C$ and returns a minimizer of $f$ on $C$ in \autoref{line: alg: select omega and d: return minimal candidate}. Hence, it remains to prove that
    \begin{equation*}
        C \cap \Psi \neq \emptyset
        \text{ if }
        \Psi \cap \partial \Phi_* \neq \emptyset
        \text{ or }
        \Psi \cap \Phi_0 \neq \emptyset
        .
    \end{equation*}
    
    Consider now the case $\Psi \cap \partial \Phi_* \neq \emptyset$. Let $(d, \omega) \in \Psi \cap \partial \Phi_*$ and define $a := \max\{l, \epsilon\}$. Then
    \begin{equation*}
        d \in \{a, u\}
        \text{ or }
        d + \omega^2 \alpha \in \{x, y\}
        \text{ or }
        \omega \in \{0, 1\}
        .
    \end{equation*}
    
    If $\omega = 1$, the definitions of $f$ and $\Phi_*$ imply
    \begin{equation*}
        \max\{a, x - \alpha\} \leq \min\{u, y - \alpha\}
    \end{equation*}
    and
    \begin{equation*}
        (d, \omega)
        =
        (\min\{\max\{a, x - \alpha, \gamma - \alpha\}, u, y - \alpha\}, 1)
        .
    \end{equation*}
    This value is included in $C$ at \autoref{line: alg: select omega and d: add candidate: case omega equals 1}.
    
    If $\alpha = 0$, $(d, \omega) \in \Psi$ implies $(d, 1) \in \Psi$ for all $(d, \omega) \in \Phi$. Hence, the case $\alpha = 0$ is covered by the previous case where $\omega = 1$. Thus, assume $\alpha \neq 0$.
    
    If $d + \omega^2 \alpha = c$ for $c \in \{x, y\}$, $d \leq c$ and $\omega = \sqrt{\tfrac{c - d}{\alpha}}$.
    Since
    \begin{equation*}
        f \left( d, \sqrt{\tfrac{c - d}{\alpha}} \right)
        =
        (c - \gamma)^2 + \left( \sqrt{\tfrac{c - d}{\alpha}} - 1 \right)^2 \beta
    \end{equation*}
    and $(d, \omega)$ is a minimizer of $f$ on $\Psi$, it follows 
    \begin{equation*}
        d \in \{c - \alpha, a, u\}
        \text{ if }
        d + \omega^2 \alpha = c
        .
    \end{equation*}
    $d = c - \alpha$ any $d + \omega^2 \alpha = c$ imply $\omega = 1$. The case $\omega = 1$ was already considered.
        
    The case $d \in \{a, u\}$ is considered now. Then $(d, \omega) \in \Phi$ is equivalent to $\omega \in [\check{\omega}_d, \hat{\omega}_d]$ with
    \begin{equation*}
        \check{\omega}_d
        := 
        \sqrt{\tfrac{\max\{x - d, 0\}}{\alpha}},~~~
        \hat{\omega}_d
        := 
        \sqrt{\tfrac{\min\{y - d, \alpha\}}{\alpha}}
        .
    \end{equation*}
    Hence, $d = u$ implies $y \geq u$ and $d = a$ implies $x - \alpha \leq a$.
    Define
    \begin{equation*}
        \Omega_d := \{ \omega \in \R ~|~ \tfrac{\partial}{\partial \omega} f(d, \omega) = 0 \}
        .
    \end{equation*}
    $\omega \in (\check{\omega}_d, \hat{\omega}_d)$ implies $\omega \in \Omega_d$. $\omega = \check{\omega}_d$ implies $\min \Omega_d \leq \check{\omega}_d$ and $\omega = \hat{\omega}_d$ implies $\max \Omega_d \geq \hat{\omega}_d$. Hence
    \begin{equation*}
        \omega
        \in
        \{ \min\{\max\{\omega, \check{\omega}_d\}, \hat{\omega}_d \} \} ~|~ \omega \in \Omega_d \}
    \end{equation*}
    These values are included in $C$ in \autoref{line: alg: select omega and d: add candidate: omega inner}.
    
    The last case is $\omega = 0$. This implies $d = a$, because $(d, \omega)$ is a minimizer of $f$ on $\Phi$. The case $d = a$ was already considered.
    
    Hence,
    \begin{equation*}
        C \cap \Psi \neq \emptyset
        \text{ if }
        \Psi \cap \partial \Phi_* \neq \emptyset
        .
    \end{equation*}
    
    Thus, it remains to show that 
    $
        C \cap \Psi \neq \emptyset
        \text{ if }
        \Psi \cap \Phi_0 \neq \emptyset
        .
    $
    Hence, consider now the case $\Psi \cap \Phi_0 \neq \emptyset$. The definition of $\Phi_0$ implies then $(0, 0) \in \Psi$ and $\max\{l, x\} \leq 0$. This implies $\epsilon \leq u, y$ due to the requirements of the algorithm. Thus, $(\epsilon, 0) \in \Phi$. Hence, since $(0, 0) \in \Psi$,
    \begin{align*}
        \gamma^2 + \beta
        =
        f(0, 0)
        \leq
        f(\epsilon, 0)
        =
        (\epsilon - \gamma)^2 + \beta
        =
        \gamma^2 - 2 \epsilon \gamma + \epsilon^2 + \beta
        .
    \end{align*}
    Thus, $\Psi \cap \Phi_0 \neq \emptyset$ implies $2 \gamma \leq \epsilon$. $(0, 0)$ is included in $C$ in \autoref{line: alg: select omega and d: add candidate: (0, 0)} in this case. Hence
    \begin{equation*}
        C \cap \Psi \neq \emptyset
    \end{equation*}
    and the algorithm returns a value in $\Psi$ in all cases.
    
    \flushright\qed
\end{proof}

\subsection{Permutation} \label{subsec: algorithms: permutation}

Another part of \nameref{alg: approximated decomposition} with some flexibility in its design is the permutation step in \autoref{line: alg: approximated decomposition: permutation} where the row and column for the current iteration are chosen. This choice drastically affects the output of \nameref{alg: approximated decomposition} and thus of \nameref{alg: approximated matrix}, too. Several strategies for the permutation are conceivable.

A strategy to reduce the approximation error is to choose the permutation that minimizes the additional approximation error. To achieve this, in each iteration the additional approximation error for all remaining indices is computed and the one with the lowest additional approximation error is chosen. If this is the same for several indices, a higher value in $d$ is preferred. As already stated in the previous subsection, this reduces the condition number of the approximation and increases the numerical stability. If these values are the same as well, a lower $\omega$ and then a lower index is preferred.

Another strategy is to prioritize higher values in $d$ instead of lower additional approximation errors. This improves the condition number and the numerical stability even further and does not necessarily increase the total approximation error as numerical experiments have shown. 

To use this strategy, \autoref{line: alg: approximated decomposition extension: minimal change: select minimum} in \nameref{alg: approximated decomposition extension: minimal change} can be replaced by the following code snippet \nameref{alg: approximated decomposition extension: minimal change with permutation: max d} presented in Algorithm \ref{alg: approximated decomposition extension: minimal change with permutation: max d}. Furthermore in \nameref{alg: approximated decomposition} , the swap in \autoref{line: alg: approximated decomposition: swap} has to be moved after \nameref{alg: approximated decomposition extension: minimal change with permutation: max d} and \autoref{line: alg: approximated decomposition: permutation} could be removed.

\begin{algorithm}[H]
\normalsize
\caption{CHOOSE\_$p$\_$d$\_$\omega$} 
\label{alg: approximated decomposition extension: minimal change with permutation: max d}

\begin{algorithmic}[1]
    \Indent
    \Indent
    
    \State $\hat{d} \gets -\infty$        
    \For {$k \gets i, \ldots, n$}            
        \State $(\tilde{d}, \tilde{\omega}) \gets$ \Call{minimal\_change}{$x_{p_{k}}, y_{p_{k}}, l, u, \epsilon, \alpha_{p_{k}}, \beta_{p_{k}}, A_{p_{k} p_{k}}$}
        \State $\tilde{f} \gets (\tilde{d} + \tilde{\omega}^2 \alpha_{p_{k}} - A_{p_{k} p_{k}})^2 + (\tilde{\omega} - 1)^2 \beta_{p_{k}}$ 
         \If {$(-\tilde{d}, \tilde{f}, \tilde{\omega}, k) < (-\hat{d}, \hat{f}, \hat{\omega}, j)$ in lexicographical order}
            \State $j \gets k$, $\hat{d} \gets \tilde{d}$, $\hat{\omega} \gets \tilde{\omega}$, $\hat{f} \gets \tilde{f}$  
        \EndIf
    \EndFor
    \State $(d_{p_j}, \omega_{p_j}) \gets (\hat{d}, \hat{\omega})$
    
    \EndIndent
    \EndIndent
\end{algorithmic}
\end{algorithm}

For sparse matrices, the permutation also affects the sparsity pattern of the matrix $L$. Hence, it would be beneficial to choose a permutation which reduces the number of nonzero values in $L$ and thus reduces also the computational effort and the memory consumption. However, minimizing the number of nonzero values is a NP-complete problem \cite{Yannakakis1981}.

However, several heuristic methods exist, which can reduce the number of nonzero values significantly. These are band reducing permutation algorithms like the Cuthill–McKee algorithm \cite{Cuthill1969} and the reverse Cuthill–McKee algorithm \cite{George1981}, symmetric approximate minimum degree permutation algorithms \cite{George1989}, like for example \cite{Amestoy1996}, or symmetric nested dissection algorithms. A good overview is provided by \cite[chapter 7]{Davis2006} and \cite[Chapter 8]{Davis2016}. It should be taken into account that only symmetric permutation methods are applicable in our context.

\subsection{Complexity} \label{subsec: algorithms: complexity}

In the context of large matrices and limited resources, the needed run time and memory of \nameref{alg: approximated matrix} and \nameref{alg: approximated decomposition} are crucial. 

The fastest way to check if $A \in \mathcal{C}^{n \times n}$ is positive definite is to try to calculate a (classical) Cholesky decomposition of $A$, that is a lower triangular matrix $L$ with $A = L L^H$ \cite[Chapter 10]{Higham2002}, \cite[Chapter 4.2]{Golub1996}. This needs at worst $\tfrac{1}{3} n^3 + \mathcal{O}(n^2)$ basic operations and stores $2 n^2 + \mathcal{O}(n)$ numbers in the real valued case. The needed memory can be reduced if only the lower triangles of $A$ and $L$ are stored. This would result in $n^2 + \mathcal{O}(n)$ numbers. It can be reduced even more if $A$ can be overwritten by $L$. This would result in $\tfrac{1}{2} n^2 + \mathcal{O}(n)$ numbers.

\nameref{alg: approximated matrix} and \nameref{alg: approximated decomposition} using \nameref{alg: approximated decomposition extension: minimal change with permutation: max d} need at worst $\tfrac{1}{3} n^3 + \mathcal{O}(n^2)$ basic operations and memory for $2 n^2 + \mathcal{O}(n)$ numbers in the real valued case, too. For this only a few small modifications are necessary which are explained below. Hence, both algorithms have asymptotically the same worst case number of basic operations and memory as an algorithm which calculates a Cholesky decomposition. Thus, their overhead is negligible and vanishes asymptotically. With some small modifications, it is also possible to overwrite the input matrix $A$ with the output matrices $L$ and $B$. Thus, \nameref{alg: approximated matrix} meets objective \ref{objective: fast computation} and \ref{objective: low memory}.

\begin{sloppypar}
For \nameref{alg: minimal change}, the number of needed basic operations and numbers that have to be stored is $\mathcal{O}(1)$. Hence, \nameref{alg: approximated decomposition extension: minimal change with permutation: max d} needs $\mathcal{O}(n)$ basic operations and stores $\mathcal{O}(1)$ numbers. If \nameref{alg: approximated decomposition extension: minimal change with permutation: max d} is used in \nameref{alg: approximated decomposition} to choose the permutation as well as $d$ and $\omega$, $\mathcal{O}(n^2)$ additional basic operations have to be performed and $\mathcal{O}(n)$ additional numbers have to be stored.
\end{sloppypar}

In \nameref{alg: approximated decomposition}, a crucial part for the number of needed operations is the calculation of $L$ in \autoref{line: alg: approximated decomposition: L definition}. Here the effort can be reduced by calculating and storing $L D^{\frac{1}{2}}$ instead of $L$ first.  After that $L$ can be calculated with an effort of $\mathcal{O}(n^2)$ basic operations. This approach results in an overall worst case number of $\tfrac{1}{3} n^3 + \mathcal{O}(n^2)$ basic operations plus the basic operations needed for the permutation and the choice of $d$ and $\omega$. Furthermore $2 n^2 + \mathcal{O}(n)$ numbers have to be stored in \nameref{alg: approximated decomposition} despite the memory needed for the choice of the permutation, $d$ and $\omega$. Hence, if \nameref{alg: approximated decomposition extension: minimal change with permutation: max d} is used, the overall worst case number of basic operations is $\tfrac{1}{3} n^3 + \mathcal{O}(n^2)$ and $2 n^2 + \mathcal{O}(n)$ numbers have to be stored.

The needed storage can be reduced by storing only one triangle of $A$ and the lower triangle of $L$ and by overwriting $A$ with $L$. This would result in $n^2 + \mathcal{O}(n)$ and $\frac{1}{2} n^2 + \mathcal{O}(n)$ numbers, respectively. However, the permutation in \nameref{alg: approximated decomposition} must be taken into account here. For this, the indexing of $A$ in \autoref{line: alg: approximated decomposition: L definition} must be suitably adapted or $A$ must be permuted. However, these modifications would not influence the $\tfrac{1}{3} n^3$ as the dominant part in the number of basic operations.

For \nameref{alg: approximated matrix}, at most $\tfrac{1}{3} n^3 + \mathcal{O}(n^2)$ basic operations plus the basic operations for choosing the permutation, $d$ and $\omega$ are needed as well. This is because for each execution of \autoref{line: alg: approximated matrix: set below diagonal without omega} in \nameref{alg: approximated matrix}, the execution of \autoref{line: alg: approximated decomposition: L definition} in \nameref{alg: approximated decomposition} is once omitted. Thus, the worst case number of needed basic operations of \nameref{alg: approximated matrix} increases only by $\mathcal{O}(n^2)$ compared to the worst case number of needed basic operations of \nameref{alg: approximated decomposition}.

At most $2 n^2+ \mathcal{O}(n)$ numbers have to be stored in \nameref{alg: approximated matrix} plus the numbers that need to be stored for choosing the permutation, $d$ and $\omega$. To achieve this, $B$ must overwrite $L$. If the strict lower triangle of $L$ is allowed to overwrite the strict lower triangle of $A$, at most $n^2+ \mathcal{O}(n)$ numbers have to be stored plus the numbers for the choice of the permutation, $d$ and $\omega$.

Hence, \nameref{alg: approximated matrix}, with the small modifications mentioned above, needs at most $\tfrac{1}{3} n^3 + \mathcal{O}(n^2)$ basic operations and stores at most $2 n^2+ \mathcal{O}(n)$ numbers if \nameref{alg: approximated decomposition extension: minimal change with permutation: max d} is used to choose the permutation, $d$ and $\omega$. It stores only $n^2+ \mathcal{O}(n)$ numbers if it is allowed to overwrite the input matrix.

It would also be possible to reduce the needed memory to $\tfrac{1}{2} n^2+ \mathcal{O}(n)$ numbers by passing only the lower triangle of the matrices $A$, $L$ and $B$. Since in this case $A$ is then no longer available after the calculation of $L$, $B$ must be calculated in the more expensive way shown in Theorem \ref{theorem: B equals PT L D LH P}. This would result in $\tfrac{1}{3} n^3 + \mathcal{O}(n^2)$ additional basic operations.

In the complex valued case, the main statement remains the same: The overhead of \nameref{alg: approximated matrix} and \nameref{alg: approximated decomposition} using \nameref{alg: approximated decomposition extension: minimal change with permutation: max d} is negligible and vanishes asymptotically compared to an algorithm which calculates a (classical) Cholesky decomposition. In a similar way, an analysis can be carried out for the case where $A$ is a sparse matrix.
 \section{Implementation and numerical experiments}
\label{sec:implementation}

An implementation of the algorithms \nameref{alg: approximated matrix} and \nameref{alg: approximated decomposition} is presented in this section together with the performed numerical experiments.

\subsection{Implementation}

The algorithms \nameref{alg: approximated matrix} and \nameref{alg: approximated decomposition} presented in Section \ref{sec:LDL_approximation} are implemented in a software library written in Python \cite{Python-3.7} called matrix-decomposition library \cite{matrix-decomposition-1.2}. Their implementation uses the \nameref{alg: minimal change} algorithm and provides both permutation algorithms described in Subsection \ref{subsec: algorithms: permutation} as well as several fill reducing permutation algorithms for sparse matrices. In addition, the library provides many more approximation and decomposition algorithms together with various other useful functions regarding matrices and its decompositions.

The library is available at github \cite{matrix-decomposition-git}. It is based on NumPy \cite{NumPy-1.17}, SciPy \cite{SciPy-1.3,Virtanen2019} and scikit-sparse \cite{scikit-sparse-0.4.4}. It was extensively tested using pytest \cite{pytest-4.4.1} and documented using Sphinx \cite{Sphinx-2.0.1}. The matrix-decomposition library and all required packages are open-source. 

They can be comfortably installed using the cross-platform package manager Conda \cite{Conda} and the Anaconda Cloud \cite{matrix-decomposition-anaconda}. Here all required packages are installed during the installation of the matrix-decomposition library. The library is also available on the Python Package Index \cite{matrix-decomposition-PyPI} and is thus installable with the standard Python package manager pip \cite{pip} as well.

\subsection{Comparison with other approximation algorithms}
\label{subsec: numerical comparison}

The \nameref{alg: approximated matrix} algorithm has been compared with other modified Cholesky algorithms based on $LDL^T$ decomposition by the resulting approximation errors and the condition numbers of the approximations. For the results presented here, we have use the Frobenius norm. However, the results using the spectral norm look similar.

The other algorithms are GMW81 \cite{Gill1981}, which is a refined version of \cite{Gill1974},  GMW1 \cite{Fang2008} and GMW2 \cite{Fang2008} which are based on GMW81, SE90 \cite{Schnabel1990} and its refined version SE99 \cite{Schnabel1999} as well as SE1 \cite{Fang2008} which in turn is based on SE99. All these algorithms are implemented in the matrix-decomposition library \cite{matrix-decomposition-1.2}. These algorithms have been extended so that the approximation can have predefined diagonal values. For this, the calculated approximation was scaled by multiplying with a suitable diagonal matrix on both sides.

\nameref{alg: approximated matrix} has been configured so that the permutation strategy which prefers high values in $D$ is used and no upper bound on the values in $D$ is applied.

Different test scenarios were used. The first three scenarios are random correlation matrices disturbed by some additive unbiased noise which should be approximated by valid correlation matrices. The random correlation matrices have been generated by the algorithm described in \cite{Davies2000}. The off-diagonal values of the symmetric noise matrices have been drawn from a normal distribution with expectation value zero and 0.1, 0.2 or 0.3 as standard deviation depending on the scenario. The diagonal values of the noise matrices were zero  in all scenarios.

\begin{figure}
    \centering
    \begin{minipage}{0.49\textwidth}
        \includegraphics[width=\linewidth,height=0.71\linewidth]{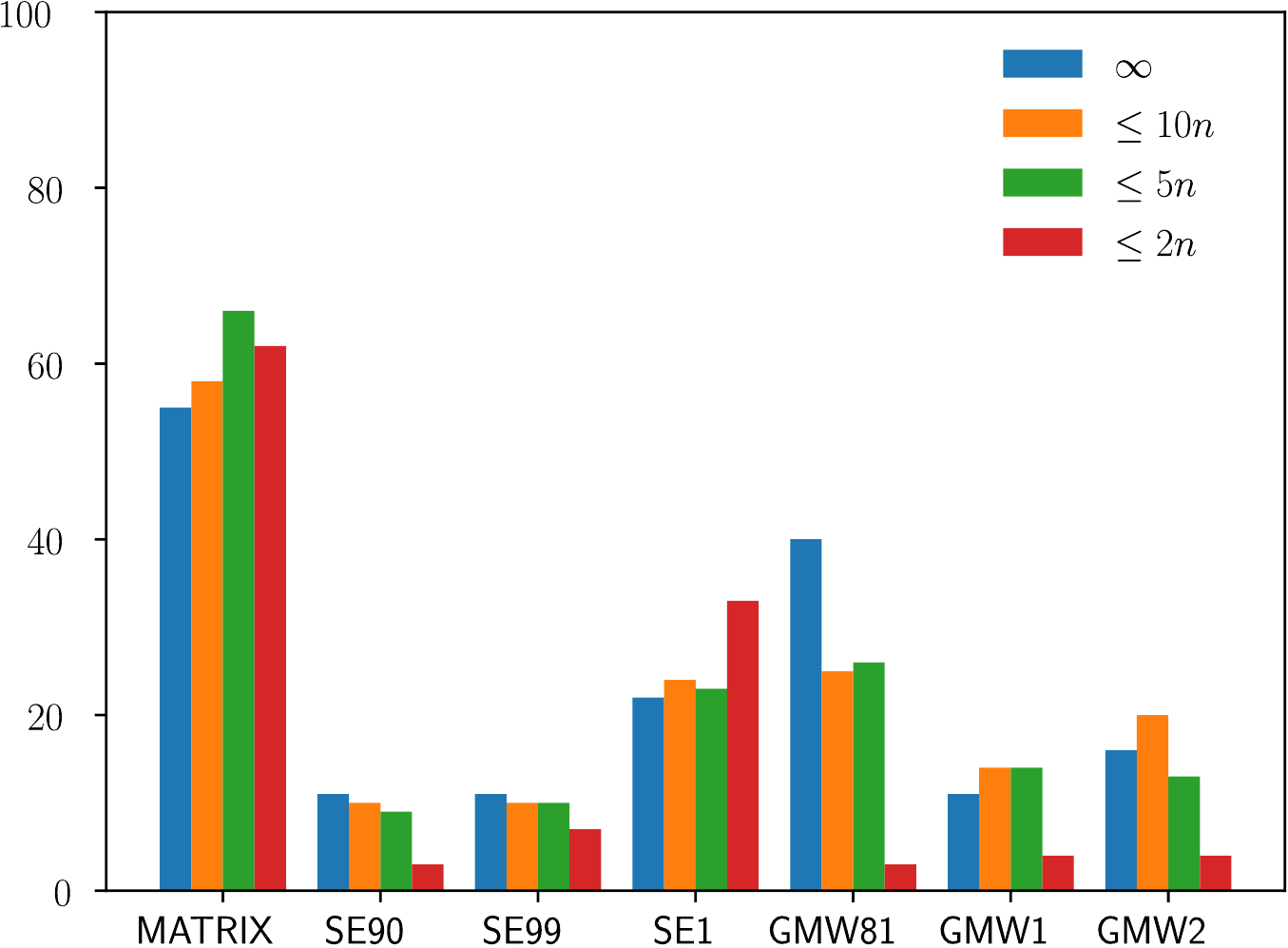}
    \end{minipage}
    \hfill
    \begin{minipage}{0.49\textwidth}
        \includegraphics[width=\linewidth,height=0.71\linewidth]{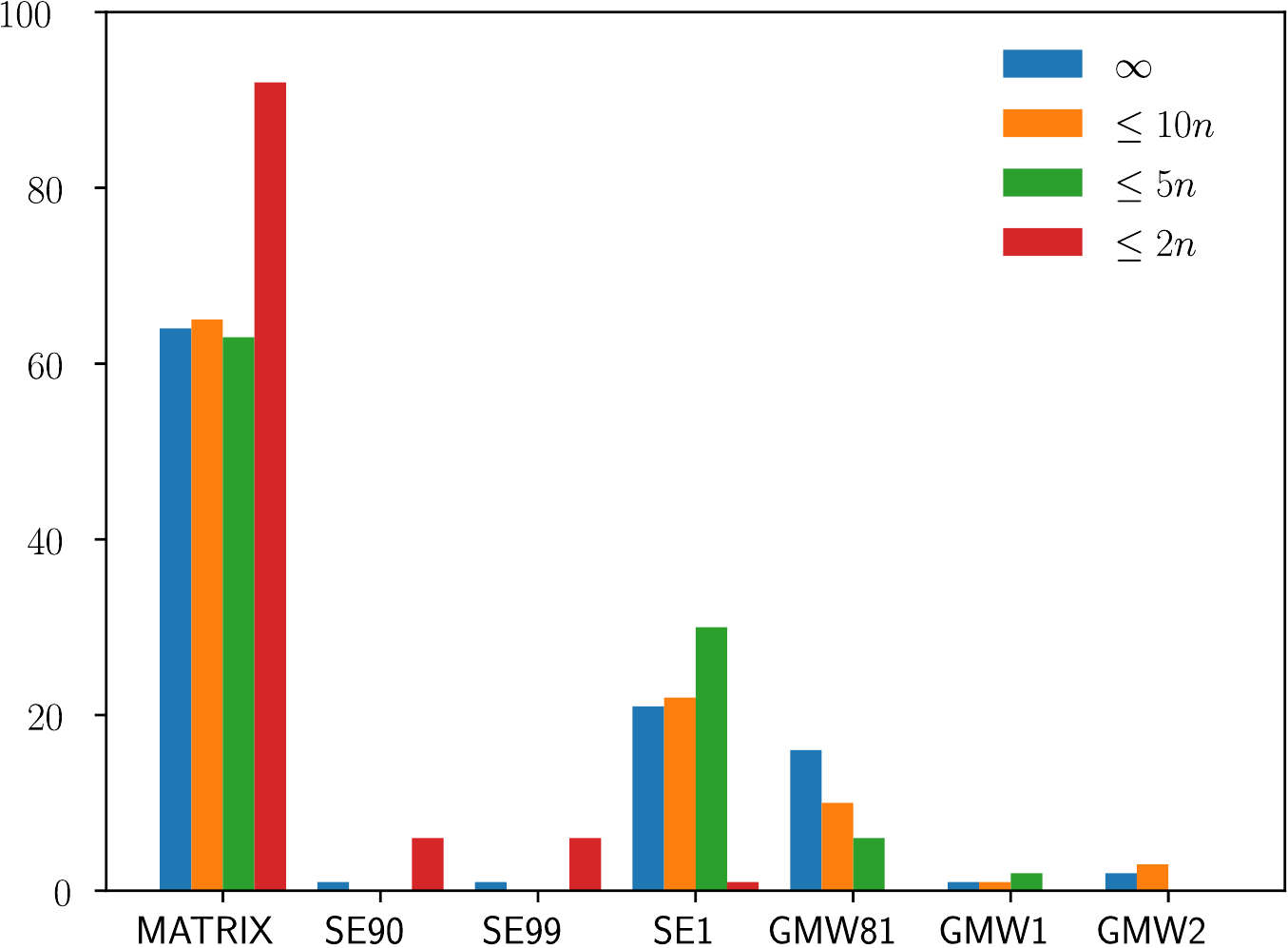}
    \end{minipage}
    \hfill
    \begin{minipage}{0.49\textwidth}
        \includegraphics[width=\linewidth,height=0.71\linewidth]{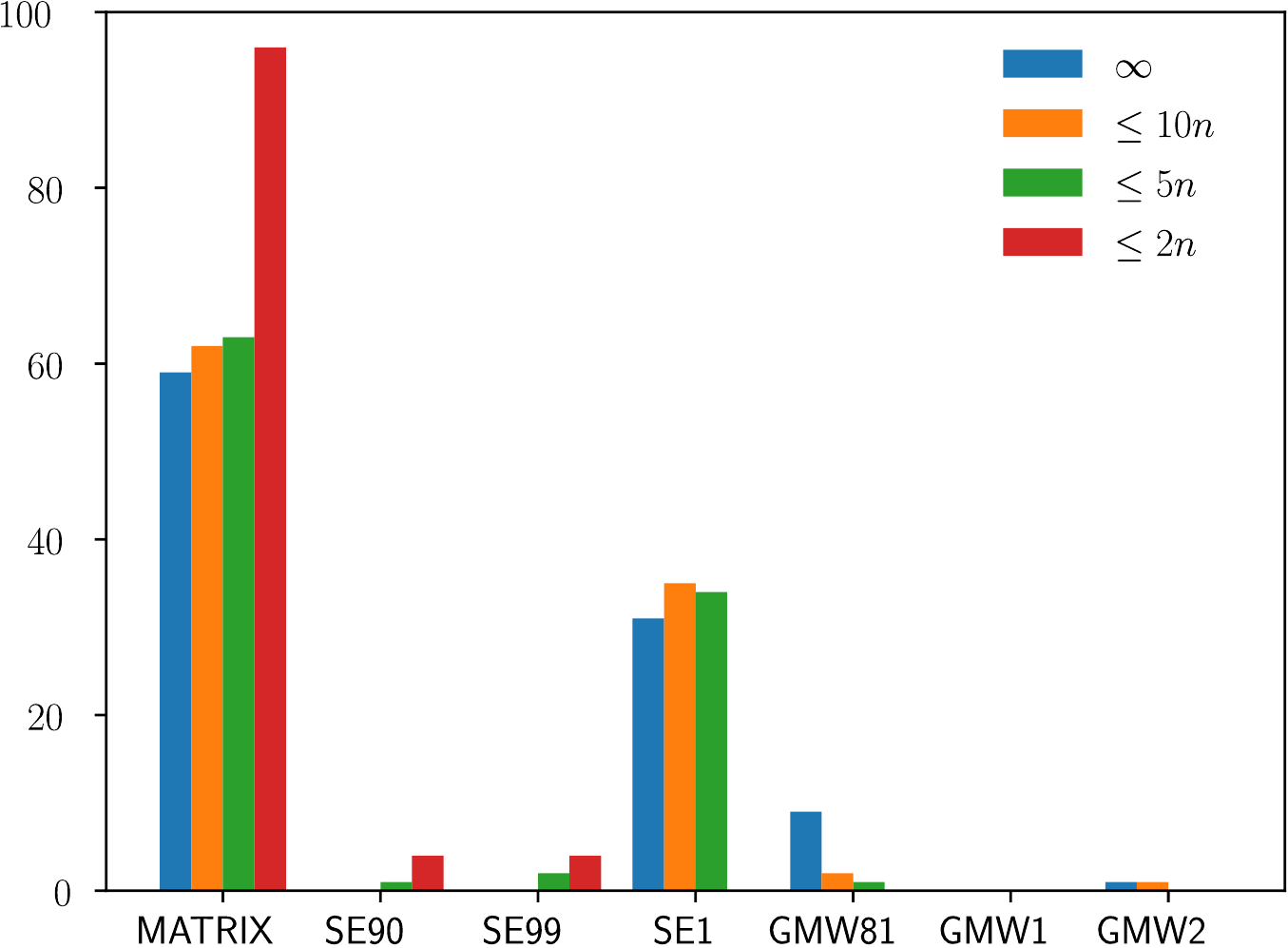}
    \end{minipage}
    \hfill
    \begin{minipage}{0.49\textwidth}
        \includegraphics[width=\linewidth,height=0.71\linewidth]{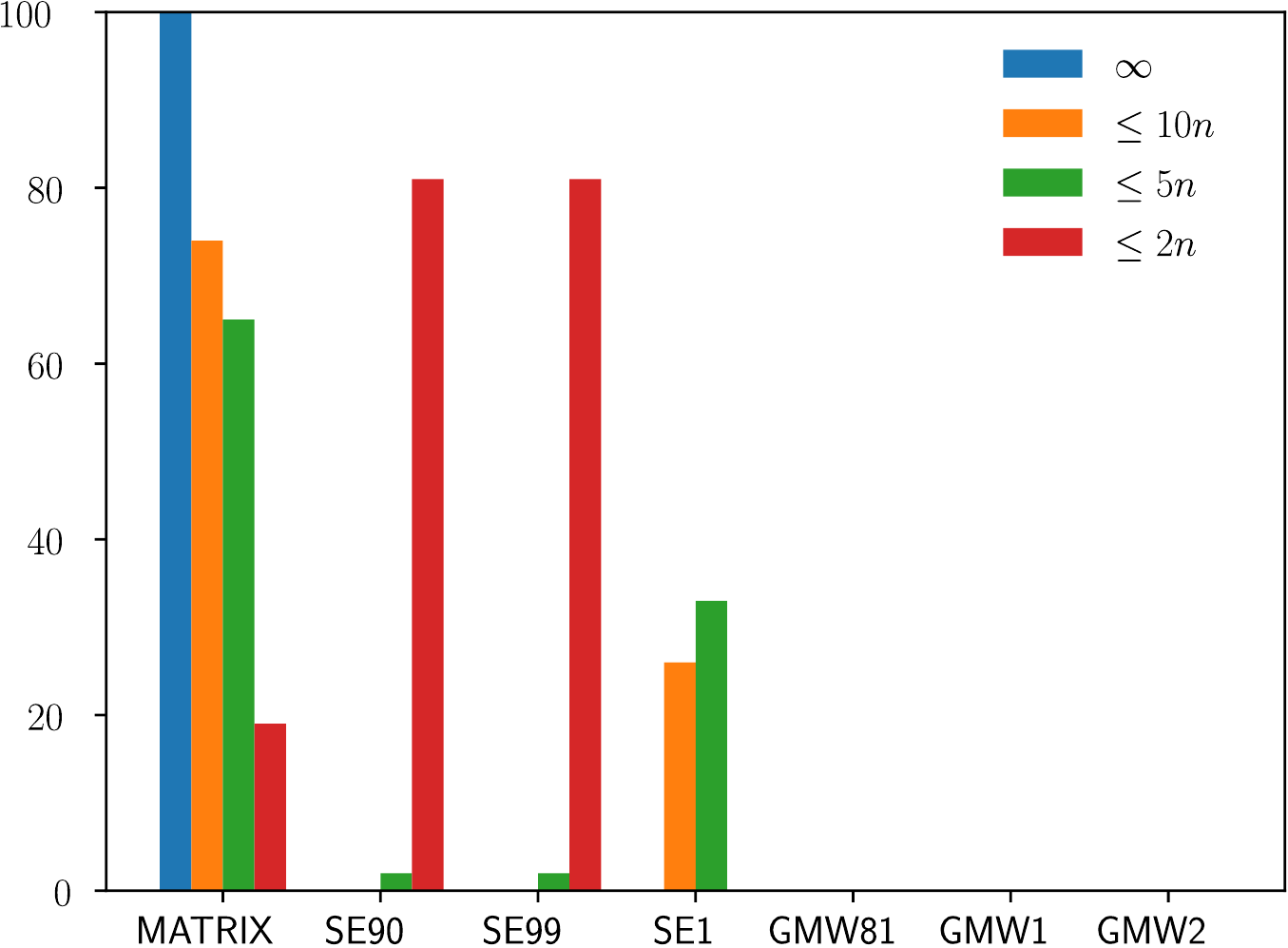}
    \end{minipage}
    \hfill
    \begin{minipage}{0.49\textwidth}
        \includegraphics[width=\linewidth,height=0.71\linewidth]{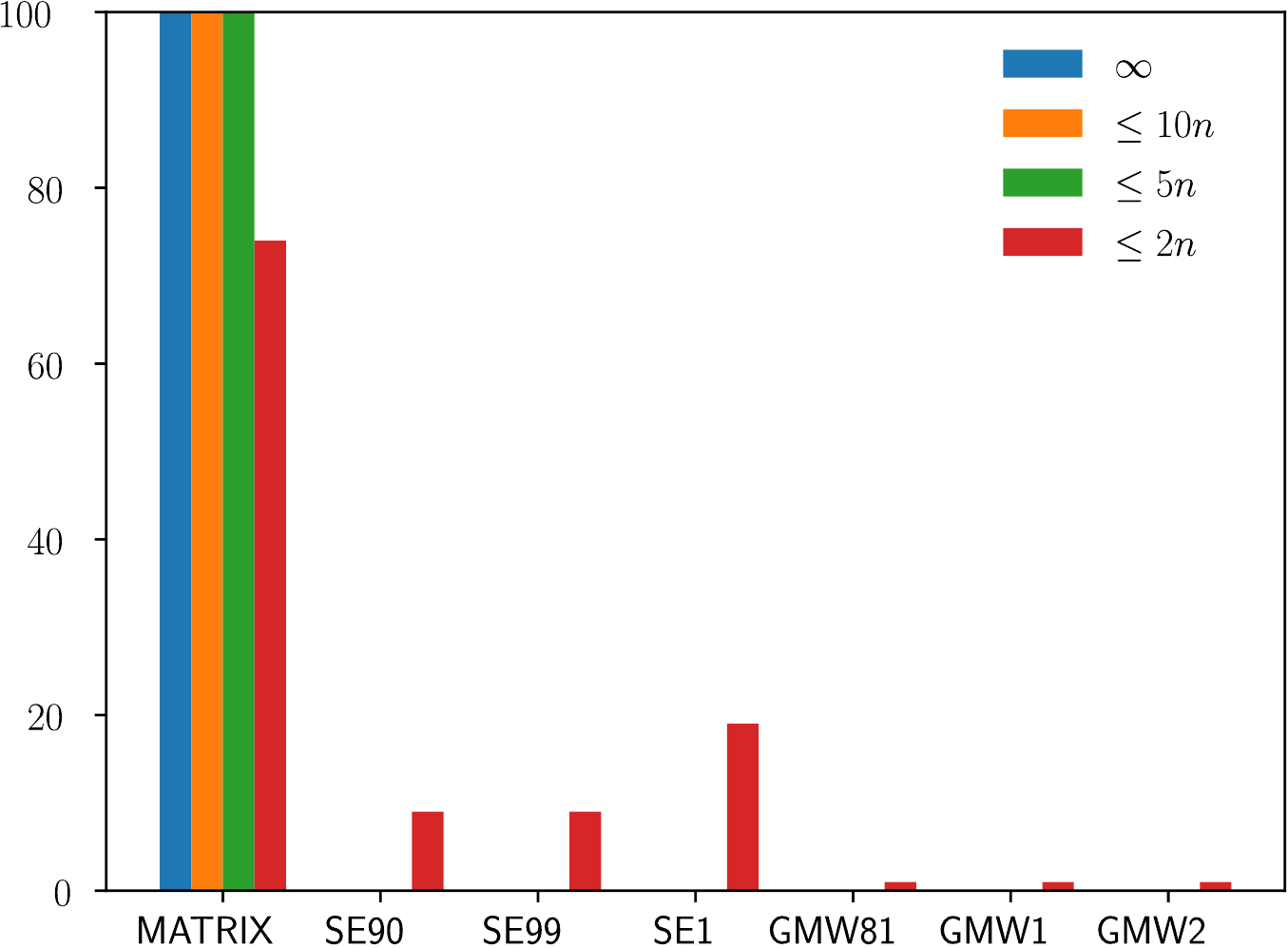}
    \end{minipage}
    \hfill
    \begin{minipage}{0.49\textwidth}
        \includegraphics[width=\linewidth,height=0.71\linewidth]{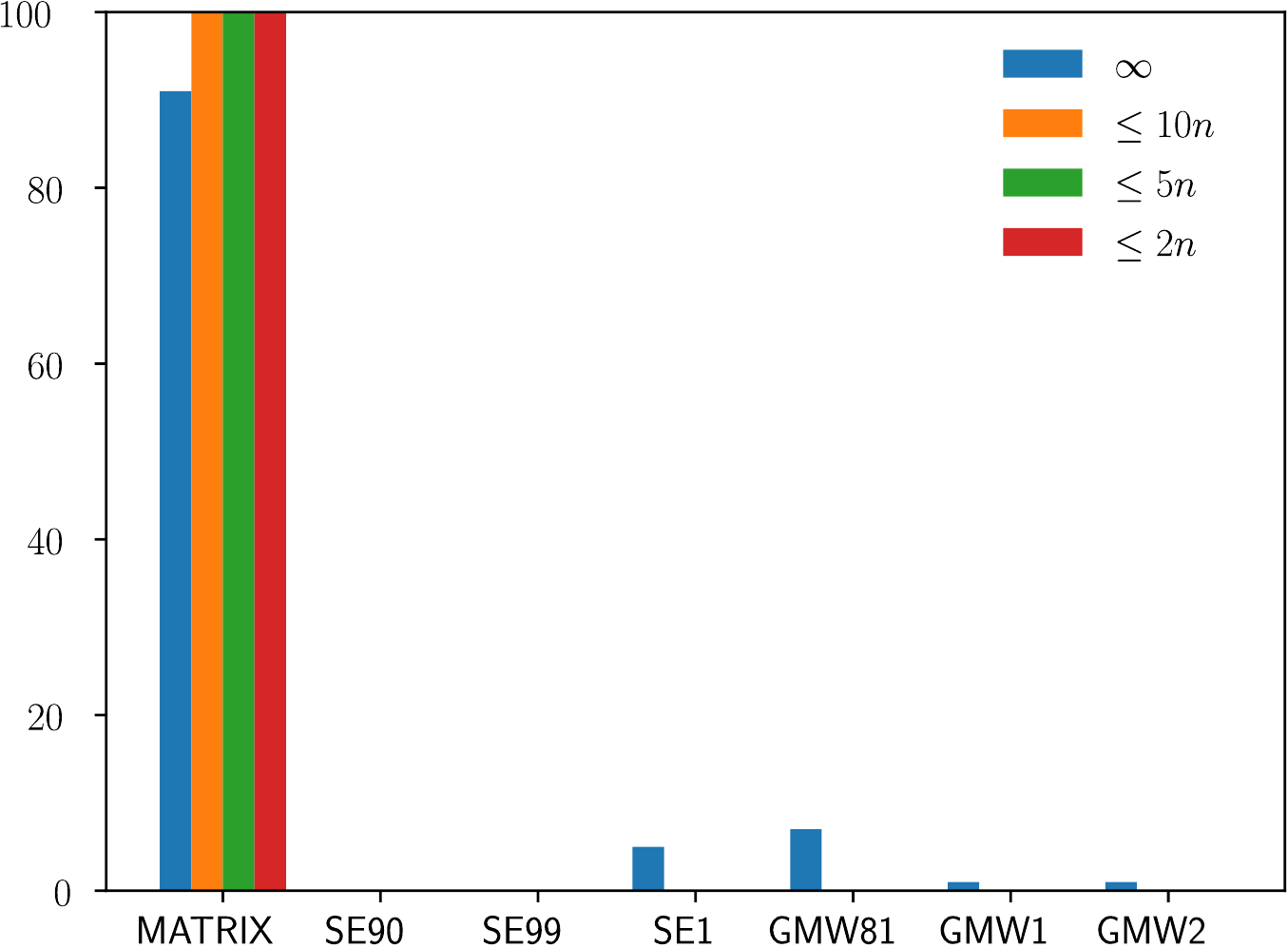}
    \end{minipage}
    \caption{Frequency, how often the algorithm achieved the smallest approximation error for the four different bounds on the condition number of the approximation (different colors) and for each of the six test scenarios (different plots).}
\label{fig: number of best}
\end{figure}

The last three scenarios are randomly generated symmetric matrices with eigenvalues uniformly distributed in $[-10^4, 10^4]$, $[-10^4, 1]$ or $[-1, 10^4]$, depending on the scenario, which should be approximated by symmetric positive semidefinite matrices. Each of these random symmetric matrices has been generated by multiplying a random orthogonal matrix, generated with the algorithm described in \cite{Stewart1980}, with a diagonal matrix with the chosen eigenvalues as diagonal values and then multiplying this with the transposed random orthogonal matrix. The eigenvalues have been drawn from uniform distributions and were altered so that each matrix has at least one negative and one positive eigenvalue.

\begin{figure}
    \centering
    \begin{minipage}{0.49\textwidth}
        \includegraphics[width=\linewidth,height=0.71\linewidth]{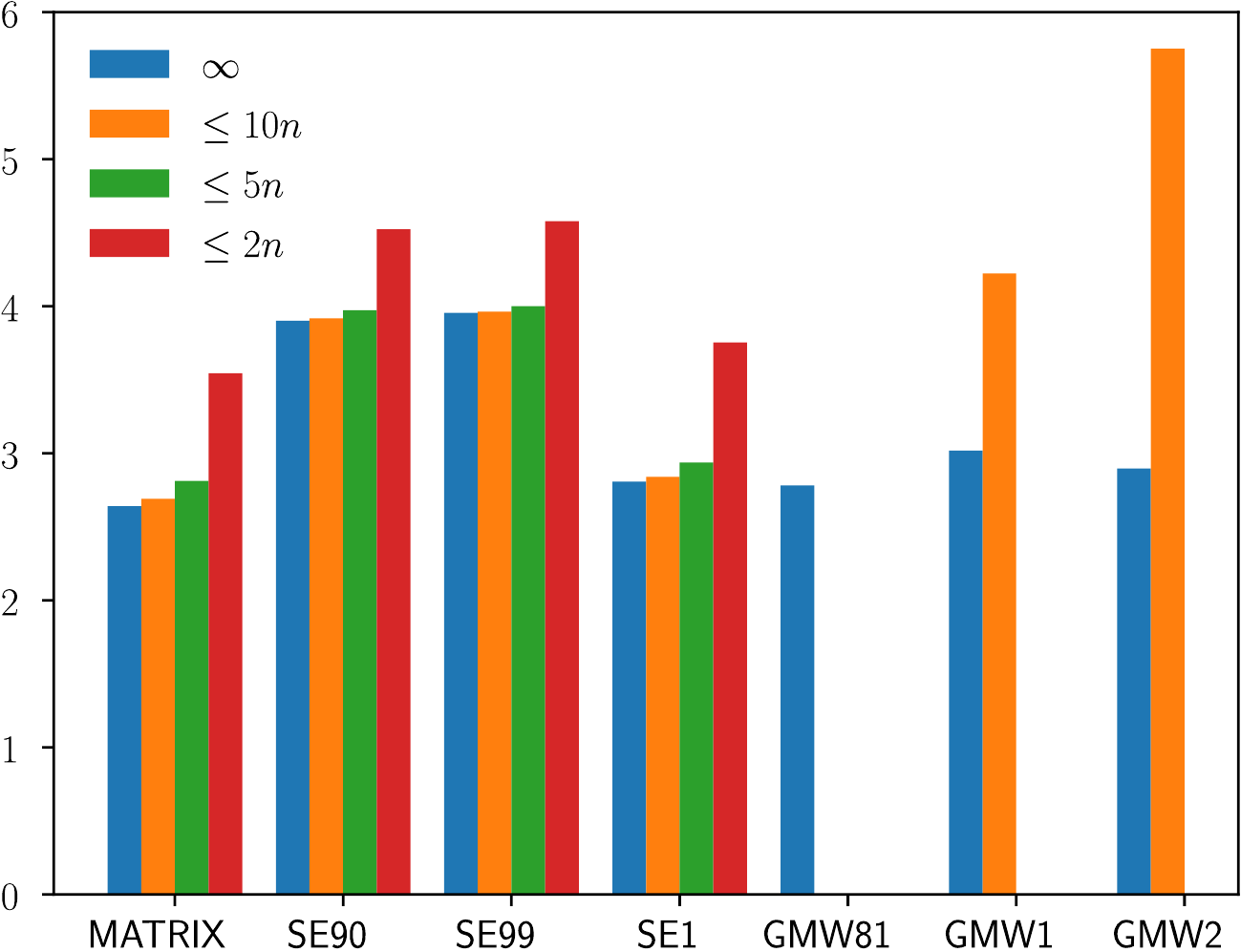}
    \end{minipage}
    \hfill
    \begin{minipage}{0.49\textwidth}
        \includegraphics[width=\linewidth,height=0.71\linewidth]{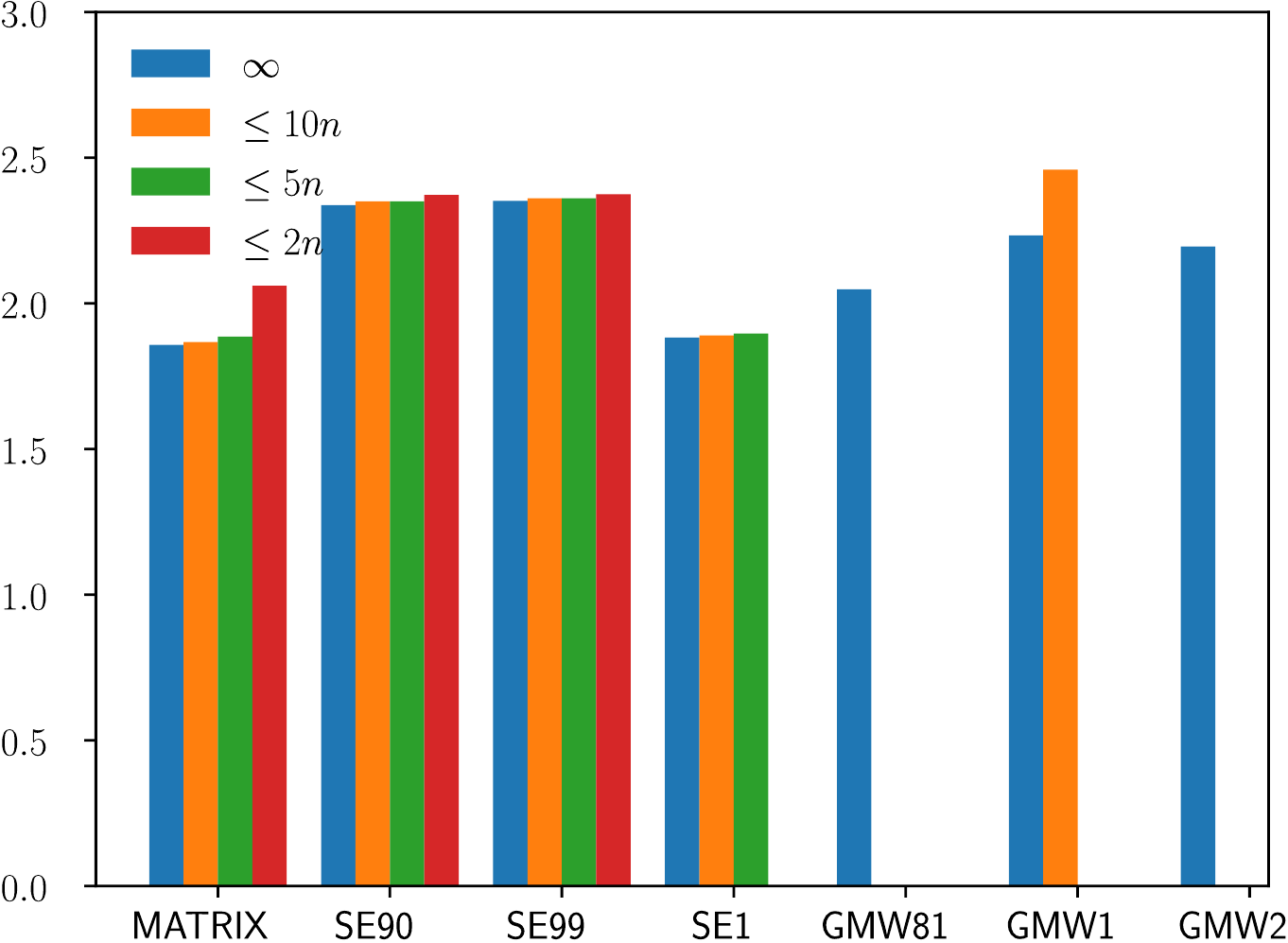}
    \end{minipage}
    \hfill
    \begin{minipage}{0.49\textwidth}
        \includegraphics[width=\linewidth,height=0.71\linewidth]{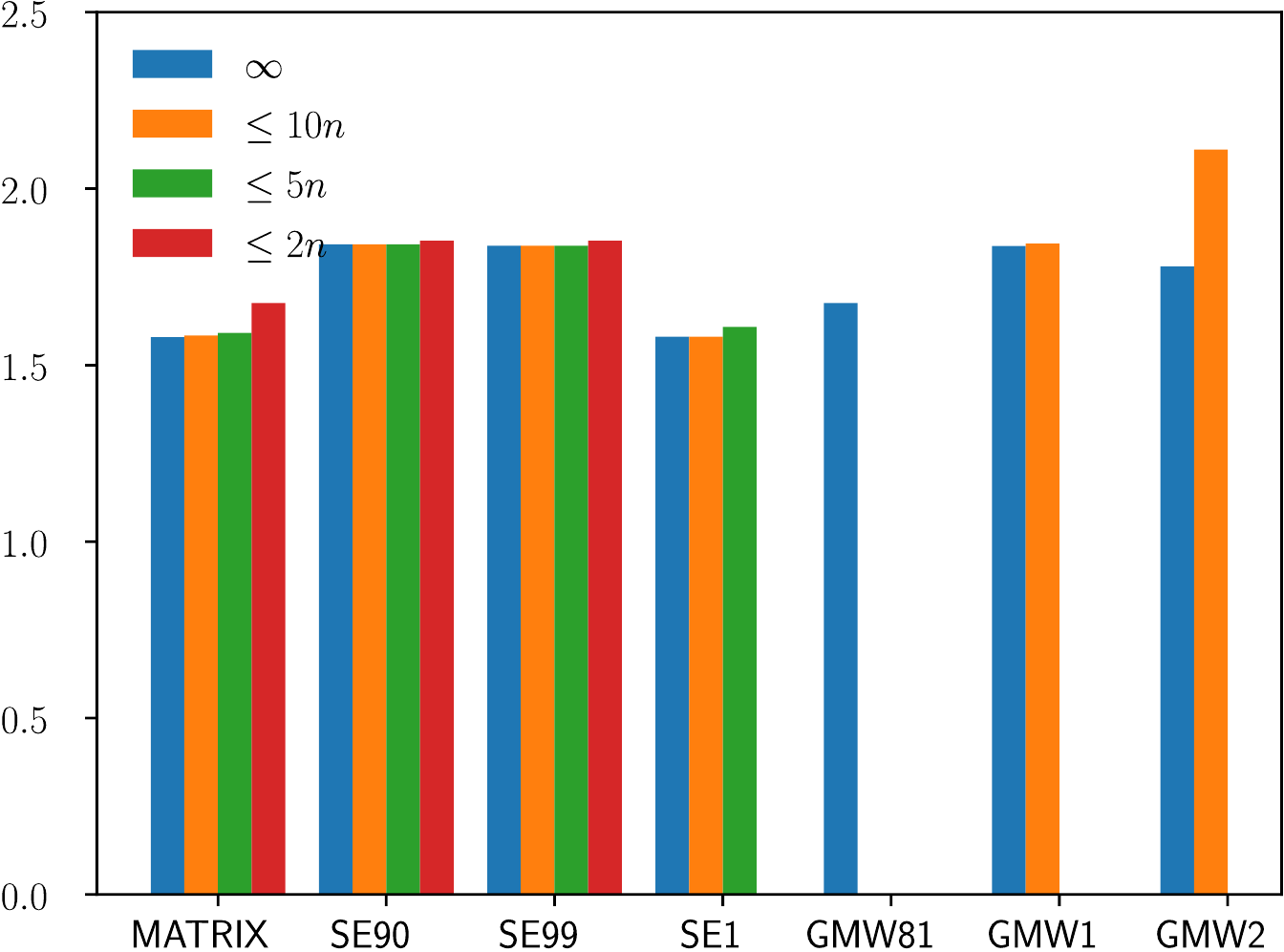}
    \end{minipage}
    \hfill
    \begin{minipage}{0.49\textwidth}
        \includegraphics[width=\linewidth,height=0.71\linewidth]{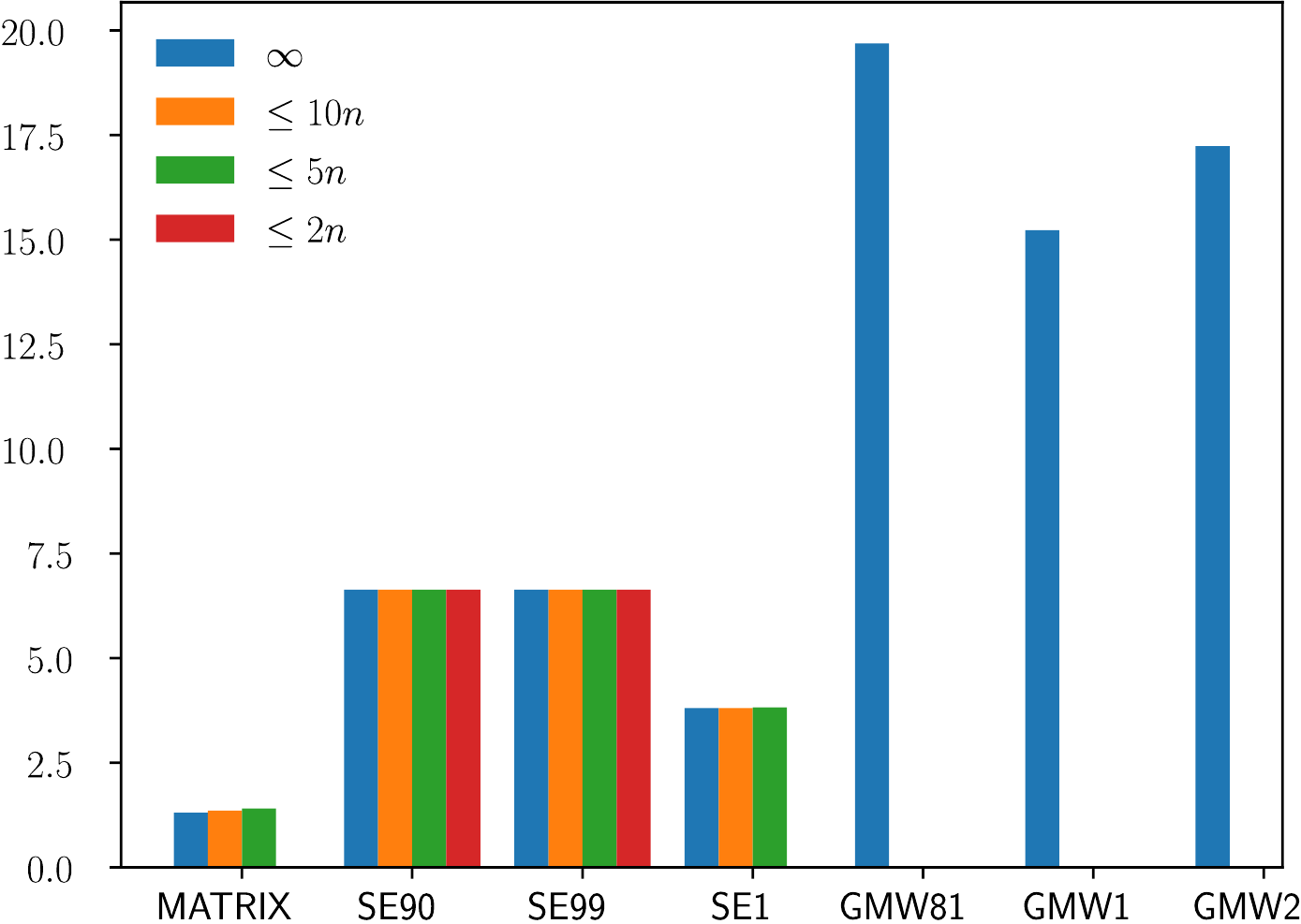}
    \end{minipage}
    \hfill
    \begin{minipage}{0.49\textwidth}
        \includegraphics[width=\linewidth,height=0.71\linewidth]{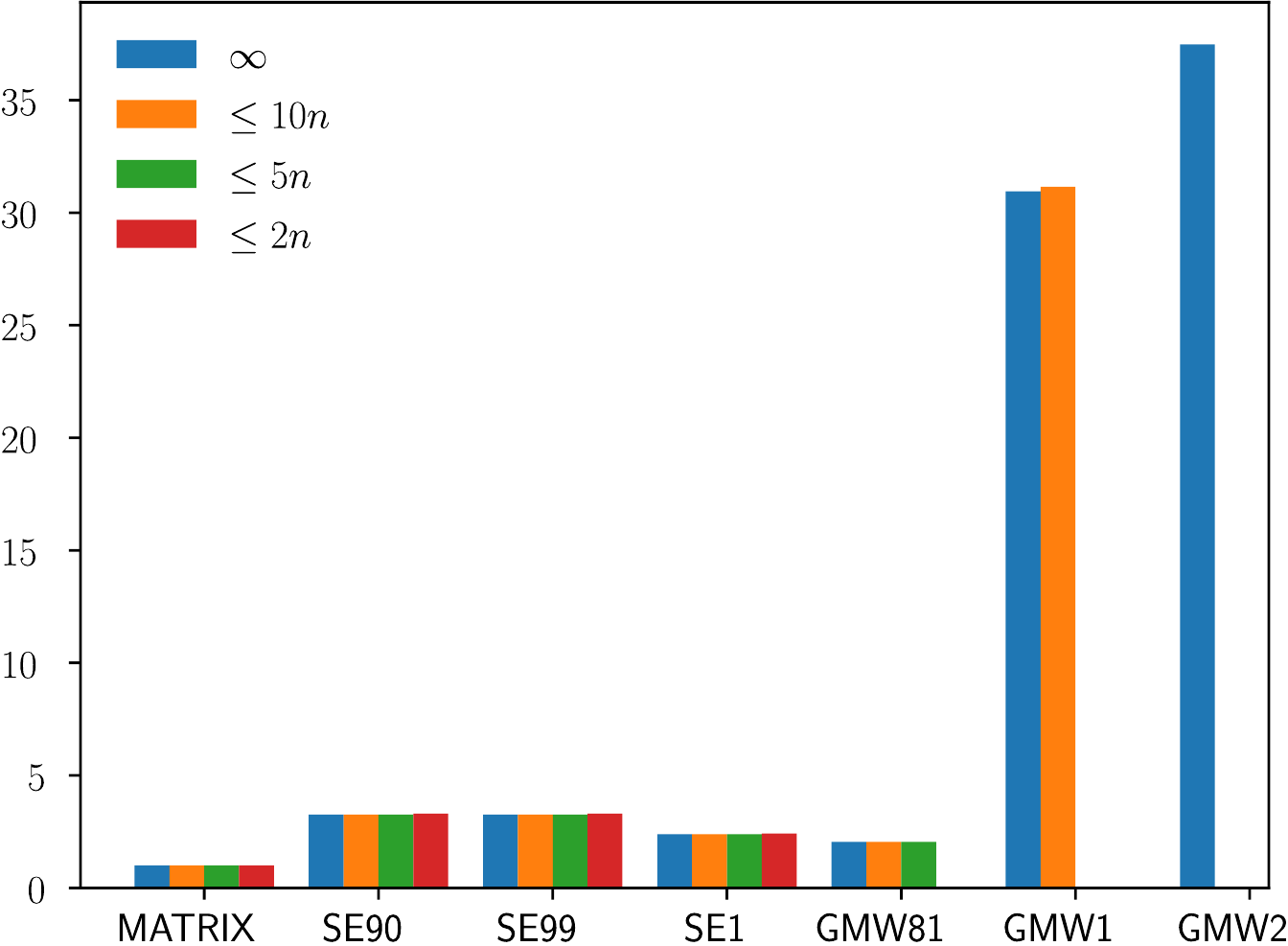}
    \end{minipage}
    \hfill
    \begin{minipage}{0.49\textwidth}
        \includegraphics[width=\linewidth,height=0.71\linewidth]{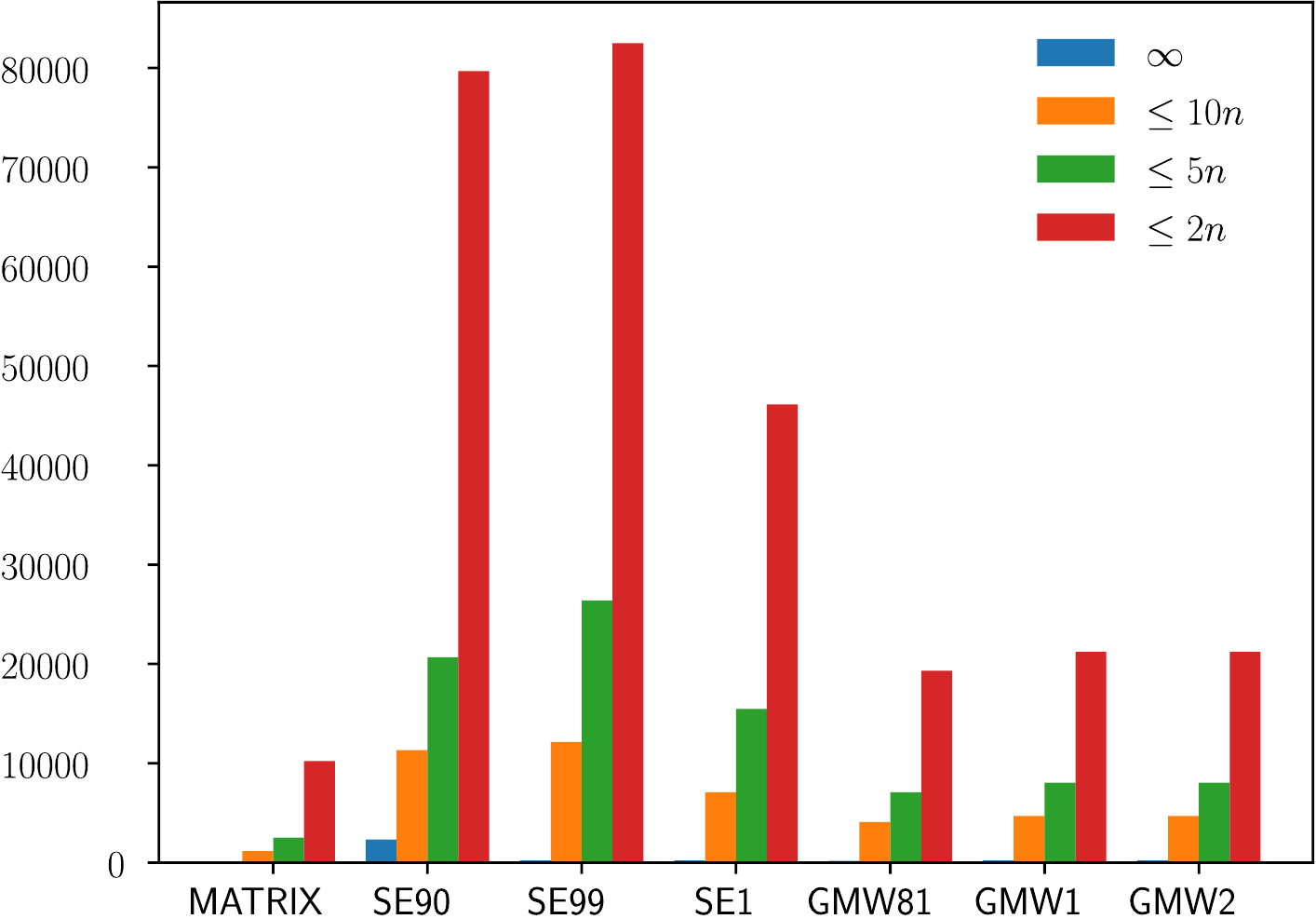}
    \end{minipage}
    \caption{Median of the approximation errors for the four different bounds on the condition number of the approximation (different colors) and for each of the six test scenarios (different plots).}
\label{fig: error}
\end{figure}

For each of the six scenarios, 100 matrices have been generated with dimensions evenly distributed between 10, 20, 30, 40 and 50 and each of them was approximated.

The approximations are assessed according to the approximation error and their condition number using different objectives. The first one is to minimize the approximation error without caring about the condition number. The other three are to minimize the approximation error while getting a condition number lower or equal to $10 n$, $5 n$ and $2 n$, respectively, where $n$ is the dimension of the matrix. This corresponds to the requirement, that the condition number should be sufficiently small (but must not be minimal), which often occurs in application examples. Minimizing only the condition number without taking the approximation error into account is not useful.

Each of the algorithms has a parameter, representing a lower bound on the values of $D$, allowing to favor a low approximation error or a low condition number. Hence, each algorithm has been executed several times with different values for this parameter and for each of the four objectives only the approximation which best meets the objective was taken into account.

Figure \ref{fig: number of best} shows how many times each algorithm has computed the approximation with the smallest approximation error among all tested algorithms for the six scenarios and the four objectives. The \nameref{alg: approximated matrix} algorithm clearly outperforms all other tested algorithms in all scenarios.

Figure \ref{fig: error} shows the median of the approximation errors for each of the six scenarios and the four objectives. The approximation errors are relative to the minimal possible approximation errors which have been calculated using the methods described in \cite{Higham1988} and \cite{Higham2002a}.  No bar in Figure \ref{fig: error} indicates that the algorithm was not able to calculate an approximation with the restriction to the condition number for at least half of the test matrices.

The results show that \nameref{alg: approximated matrix} calculates approximations with approximation errors usually close to optimal and still sufficiently small condition numbers. In addition, it performs better, sometimes very considerably, than the other tested algorithms.

The numerical tests have also indicated that, for determining $d_i$, a varying lower bound $\hat{l}_i$, defined as
\begin{equation*}
    \hat{l}_i
    :=
    \begin{cases}
        l &\text{ if } \tfrac{1}{2} c_{p_i} < l\\
        u &\text{ if } \tfrac{1}{2} c_{p_i} > u\\
        \tfrac{1}{2} c_{p_i} & \text{ else}
    \end{cases}
    \text{ with }
    c_i
    :=
    \begin{cases}
        x_i &\text{ if } A_{i i} < x_i\\
        y_i &\text{ if } A_{i i} > y_i\\
        A_{i i} & \text{ else}
    \end{cases}
\end{equation*}
for each $i \in \{1 \ldots, n\}$,
is useful in order to achieve a low approximation error and a low condition number. This varying lower bound is also choosable in the matrix-decomposition library.

 \section{Conclusions}
\label{sec: summary}

A new algorithm to approximate Hermitian matrices by positive semidefinite Hermitian matrices was presented. In contrast to existing algorithms, it allows to specify bounds on the diagonal values of the approximation.

It tries to minimize the approximation error in the Frobenius norm and the condition number of the approximation. Parameters of the algorithm can be used to select which of these two objectives is preferred if not both objectives can be meet equally well. Numerical tests have shown that the algorithms outperforms existing algorithms regarding the approximation error as well as the condition number.

The algorithm is suitable for very large matrices, since it needs only $\tfrac{1}{3} n^3 + \mathcal{O}(n^2)$ basic operations and storage for $n^2 + \mathcal{O}(n)$ numbers in the real valued case. This is asymptotically the same number of basic operations as the computation of a Cholesky decomposition would need. Moreover the algorithm is also suitable for sparse matrices since it preserves the sparsity pattern of the original matrix.  

The $LDL^H$ decomposition of the approximation is calculated as a byproduct. This allows to solve corresponding linear equations or to calculate the corresponding determinant very quickly. If such a decomposition should be calculated anyway, the algorithm has no significant overhead.

Two parts in the algorithm are realizable in many different ways. Various possibilities were presented, more are conceivable.

An open-source implementation of this algorithm is freely available. The implementation is fully documented and easy to install. Extensive numerical tests confirm the functionality of the algorithm and its implementation.

Numerical optimization and statistics are two fields of application in which the algorithm can be of particular interest.
 
\appendix
\section{Appendix}
\label{sec:LDL_decomposition}

\begin{theorem} \label{theorem: LDLH decomposition: existance and uniqueness}
    Let $A \in \C^{n \times n}$ be a positive semidefinite matrix. $A$ has a $L D L^H$ decomposition. If $A$ is positive definite this decomposition is unique.
\end{theorem}
\begin{proof}
    See \cite[p. 13]{Householder1964}.
    \flushright\qed
\end{proof}

\begin{theorem} \label{theorem: LDLH decomposition and eigenvalues}
    Let $L \in \C^{n \times n}$ be a lower triangular matrix with ones on the diagonal and $D \in \R^{n \times n}$ a diagonal matrix. $L D L^H$ is 
    \begin{enumerate}[label=\alph*)]
        \item invertible if and only if $D_{ii} \neq 0$ for all $i \in \{1, \ldots, n\}$.
        \item positive semidefinite if and only if $D_{ii} \geq 0$ for all $i \in \{1, \ldots, n\}$ 
        \item positive definite if and only if $D_{ii} > 0$ for all $i \in \{1, \ldots, n\}$.
    \end{enumerate}
\end{theorem}

\begin{proof}
    Sylvester's law of inertia \cite{Sylvester1852} extended to Hermitian matrices \cite{Ikramov2001} implies that $L D L^H$ and $D$ have the same number of negative, zero, and respectively positive eigenvalues. Since $D$ is a diagonal matrix, the eigenvalues of $D$ are its diagonal values. Hence $L D L^H$ is invertible, positive semidefinite or positive definite if and only if the diagonal values of $D$ are non-zero, non-negative or positive, respectively. 
    \flushright\qed
\end{proof}

\begin{theorem} \label{theorem: condition number inequality: positive definite}
    Let $A \in \C^{n \times n}$ be a positive definite matrix. Let $L$ and $D$ be the matrices of its $LDL^H$ decomposition. Then   
    \begin{equation*}
        \left( \frac{\trace(A)}{n \beta} \right)^\frac{n}{2(n-1)}
        \leq
        \kappa_2(L)
        \leq
        2 \left( \frac{\trace(A)}{n \alpha} \right)^\frac{n}{2}
        ,
    \end{equation*}
    \begin{equation*}
        \kappa_2(D)
        =
        \frac{\beta}{\alpha}
        \text{ and }
        \kappa_2(A)
        \leq
        4 \frac{\beta}{\alpha} \left( \frac{\trace(A)}{n \alpha} \right)^{n}
    \end{equation*}
    \begin{equation*}
        \text{with }
        \alpha := \min_{i=1,\ldots,n} D_{ii}
        \text{ and }
        \beta := \max_{i=1,\ldots,n} D_{ii}
        .
    \end{equation*}
\end{theorem}

\begin{proof}
    Define $B := L L^H$. The definition of $B$ implies
    \begin{equation*}
        \kappa_2(L) = \sqrt{\kappa_2(B)}
    \end{equation*}
    since $\kappa_2(B) = \kappa_2(L L^H) = \kappa_2(L)^2$.

     $L$ is a lower triangular matrix with ones on the diagonal. Hence, $\det(L) = 1$ and
    \begin{equation*}
        \det(B) = \det(L) \det(L^H) = \det(L)^2 = 1
        .
    \end{equation*}
    Thus, \cite{Davis1961} state that
    \begin{equation} \label{equation: LDLH decomposition: theorem: condition number inequality: B condition number inequality}
        c^{-\frac{1}{n-1}}
        \leq
        \kappa_2(B)
        \leq
        \frac
        {1 + \sqrt{1- c}}
        {1 - \sqrt{1- c}}
        \text{ ~with~ }
        c := \left( \frac{n}{\trace(B)} \right)^n
        .
    \end{equation}
    Besides,
    \begin{equation} \label{equation: LDLH decomposition: theorem: condition number inequality: trace B}
        \trace(B)
        =
        \trace(L L^H)
        =
        \sum\limits_{i,j=1}^{n} L_{ij} \conj{L_{ij}}
        =
        \sum\limits_{i,j=1}^{n} |L_{ij}|^2
        =
        \|L\|_F^2
        .
    \end{equation}
    and
    \begin{equation*}
        \|L\|_F^2 = \sum\limits_{i,j=1}^{n} |L_{ij}|^2 \geq \sum\limits_{i=1}^{n} |L_{ii}|^2 = n
        .
    \end{equation*}
    Hence $0 \leq c \leq 1$, which implies
    \begin{equation} \label{equation: LDLH decomposition: theorem: condition number inequality: frac inequality}
        \frac
        {1 + \sqrt{1- c}}
        {1 - \sqrt{1- c}}
        =
        \frac
        {(1 + \sqrt{1- c})^2}
        {(1 - \sqrt{1- c})(1 + \sqrt{1- c})}
        =
        \frac
        {(1 + \sqrt{1- c})^2}
        {c}
        \leq
        \frac
        {2^2}
        {c}
        .
    \end{equation}
    Equation \eqref{equation: LDLH decomposition: theorem: condition number inequality: B condition number inequality}, \eqref{equation: LDLH decomposition: theorem: condition number inequality: trace B} and \eqref{equation: LDLH decomposition: theorem: condition number inequality: frac inequality} result in
    \begin{equation*}
        \left( \frac{\|L\|_F^2}{n} \right)^\frac{n}{n-1}
        \leq
        \kappa_2(B)
        \leq
        4 \left( \frac{\|L\|_F^2}{n} \right)^n
    \end{equation*}
    and thus
    \begin{equation} \label{equation: LDLH decomposition: theorem: condition number inequality: L condition number inequality}
        \left( \frac{\|L\|_F^2}{n} \right)^\frac{n}{2(n-1)}
        \leq
        \kappa_2(L)
        \leq
        2 \left( \frac{\|L\|_F^2}{n} \right)^\frac{n}{2}
        .
    \end{equation}
    
    Theorem \ref{theorem: LDLH decomposition and eigenvalues} implies $0 < \alpha$ because $A$ is positive definite. Moreover, $\alpha \leq D_{ii} \leq \beta$ for all $i \in \{1,\ldots,n \}$ by definition of $\alpha$ and $\beta$. Thus
    \begin{align*}
        \frac{\trace(A)}{\beta} 
        &=
        \frac{1}{\beta} \sum\limits_{i=1}^{n} A_{ii}
        =
        \frac{1}{\beta} \sum\limits_{i=1}^{n} \sum\limits_{j=1}^{n} L_{ij} D_{jj} \conj{L}_{ij}
        =
        \sum\limits_{i=1}^{n} \sum\limits_{j=1}^{n} |L_{ij}|^2 \frac{D_{jj}}{\beta}
        \\ &\leq
        \sum\limits_{i=1}^{n} \sum\limits_{j=1}^{n} |L_{ij}|^2
        =
        \| L \|_F^2
        =
        \sum\limits_{i=1}^{n} \sum\limits_{j=1}^{n} |L_{ij}|^2
        \\ &\leq
        \sum\limits_{i=1}^{n} \sum\limits_{j=1}^{n} |L_{ij}|^2 \frac{D_{jj}}{\alpha}
        =
        \frac{1}{\alpha} \sum\limits_{i=1}^{n} \sum\limits_{j=1}^{n} L_{ij} D_{jj} \conj{L}_{ij}
        =
        \frac{1}{\alpha} \sum\limits_{i=1}^{n} A_{ii}
        =
        \frac{\trace(A)}{\alpha} 
        .
    \end{align*}
    Hence
    \begin{equation*}
        \left( \frac{\trace(A)}{n \beta} \right)^\frac{n}{2(n-1)}
        \leq
        \kappa_2(L)
        \leq
        2 \left( \frac{\trace(A)}{n \alpha} \right)^\frac{n}{2}
    \end{equation*}
    with \eqref{equation: LDLH decomposition: theorem: condition number inequality: L condition number inequality}.
        
    Furthermore
    \begin{equation*}
        \kappa_2(D)
        =
        \frac{\max\limits_{i=1,\ldots,n} |D_{ii}|}{\min\limits_{i=1,\ldots,n} |D_{ii}|}
        =
        \frac{\beta}{\alpha}
    \end{equation*}
    since $D$ is a diagonal matrix. Thus
    \begin{align*}
        \kappa_2(A)
        &=
        \kappa_2(L D L^H)
        \leq
        \kappa_2(L) \kappa_2(D) \kappa_2(L^H)
        \\&=
        \kappa_2(L)^2 \kappa_2(D)
        \leq
        4 \frac{\beta}{\alpha} \left( \frac{\|L\|_F^2}{n} \right)^n
        ,
    \end{align*}
    because $\kappa_2(A B) \leq \kappa_2(A) \kappa_2(B)$ and $\kappa_2(A) = \kappa_2(A^H)$ for every invertible matrices $A, B \in \C^{n \times n}$. 
    \flushright\qed
\end{proof}
 
{\small

\bibliographystyle{spmpsci}
\bibliography{article_expanded}

\begin{thebibliography}{10}
\providecommand{\url}[1]{{#1}}
\providecommand{\urlprefix}{URL }
\expandafter\ifx\csname urlstyle\endcsname\relax
  \providecommand{\doi}[1]{DOI~\discretionary{}{}{}#1}\else
  \providecommand{\doi}{DOI~\discretionary{}{}{}\begingroup
  \urlstyle{rm}\Url}\fi

\bibitem{Amestoy1996}
Amestoy, P.R., Davis, T.A., Duff, I.S.: {An Approximate Minimum Degree Ordering
  Algorithm}.
\newblock SIAM Journal on Matrix Analysis and Applications \textbf{17}(4),
  886--905 (1996).
\newblock \doi{10.1137/S0895479894278952}

\bibitem{Conda}
Anaconda, I.: {Conda Package Manager}.
\newblock \urlprefix\url{https://conda.io/docs/index.html}

\bibitem{Borsdorf2010}
Borsdorf, R., Higham, N.J.: {A Preconditioned Newton Algorithm for the Nearest
  Correlation Matrix}.
\newblock IMA J. Numer. Anal. \textbf{30}(1), 94--107 (2010).
\newblock \doi{10.1093/imanum/drn085}

\bibitem{Borsdorf2010a}
Borsdorf, R., Higham, N.J., Raydan, M.: {Computing a Nearest Correlation Matrix
  with Factor Structure}.
\newblock SIAM J. Matrix Anal. Appl. \textbf{31}(5), 2603--2622 (2010).
\newblock \doi{10.1137/090776718}

\bibitem{Sphinx-2.0.1}
Brandl, G., et~al.: {Sphinx: Python documentation generator} (2019).
\newblock \urlprefix\url{www.sphinx-doc.org}.
\newblock Version 2.0.1

\bibitem{Chen2010}
Chen, Y., Wiesel, A., Eldar, Y.C., Hero, A.O.: {Shrinkage Algorithms for MMSE
  Covariance Estimation}.
\newblock IEEE Transactions on Signal Processing \textbf{58}(10), 5016--5029
  (2010).
\newblock \doi{10.1109/TSP.2010.2053029}

\bibitem{Cheng1998}
Cheng, S., Higham, N.: {A Modified Cholesky Algorithm Based on a Symmetric
  Indefinite Factorization}.
\newblock SIAM Journal on Matrix Analysis and Applications \textbf{19}(4),
  1097--1110 (1998).
\newblock \doi{10.1137/S0895479896302898}

\bibitem{Chong2013}
Chong, E., Zak, S.: {An Introduction to Optimization}, 4th edn.
\newblock Wiley Series in Discrete Mathematics and Optimization. Wiley (2013)

\bibitem{Cuthill1969}
Cuthill, E., McKee, J.: {Reducing the Bandwidth of Sparse Symmetric Matrices}.
\newblock In: Proceedings of the 1969 24th National Conference, ACM '69, pp.
  157--172. ACM, New York, NY, USA (1969).
\newblock \doi{10.1145/800195.805928}

\bibitem{Davies2000}
Davies, P.I., Higham, N.J.: {Numerically Stable Generation of Correlation
  Matrices and Their Factors}.
\newblock BIT Numerical Mathematics \textbf{40}(4), 640--651 (2000).
\newblock \doi{10.1023/A:1022384216930}.
\newblock \urlprefix\url{https://doi.org/10.1023/A:1022384216930}

\bibitem{Davis1961}
Davis, P.J., Haynsworth, E.V., Marcus, M.: {Bound for the P-condition number of
  matrices with positive roots}.
\newblock J. Res. Natl. Bur. Stand. B \textbf{65}, 13--14 (1961)

\bibitem{Davis2006}
Davis, T.: {Direct Methods for Sparse Linear Systems}.
\newblock Society for Industrial and Applied Mathematics (2006).
\newblock \doi{10.1137/1.9780898718881}

\bibitem{Davis2016}
Davis, T.A., Rajamanickam, S., Sid-Lakhdar, W.M.: A survey of direct methods
  for sparse linear systems.
\newblock Acta Numerica \textbf{25}, 383--566 (2016).
\newblock \doi{10.1017/S0962492916000076}

\bibitem{Devlin1975}
Devlin, S.J., Gnanadesikan, R., Kettenring, J.R.: Robust estimation and outlier
  detection with correlation coefficients.
\newblock Biometrika \textbf{62}(3), 531--545 (1975)

\bibitem{Fang2008}
Fang, H.r., O'Leary, D.P.: {Modified Cholesky algorithms: a catalog with new
  approaches}.
\newblock Mathematical Programming \textbf{115}(2), 319--349 (2008).
\newblock \doi{10.1007/s10107-007-0177-6}

\bibitem{Fisher2011}
Fisher, T.J., Sun, X.: {Improved Stein-type shrinkage estimators for the
  high-dimensional multivariate normal covariance matrix}.
\newblock Computational Statistics {\&} Data Analysis \textbf{55}(5),
  1909--1918 (2011).
\newblock \doi{10.1016/j.csda.2010.12.006}.
\newblock
  \urlprefix\url{http://www.sciencedirect.com/science/article/pii/S0167947310004743}

\bibitem{George1981}
George, A., Liu, J.W.: {Computer Solution of Large Sparse Positive Definite}.
\newblock Prentice Hall Professional Technical Reference (1981)

\bibitem{George1989}
George, A., Liu, J.W.: {The Evolution of the Minimum Degree Ordering
  Algorithm}.
\newblock SIAM Review \textbf{31}(1), 1--19 (1989).
\newblock \doi{10.1137/1031001}

\bibitem{Gill1974}
Gill, P.E., Murray, W.: Newton-type methods for unconstrained and linearly
  constrained optimization.
\newblock Mathematical Programming \textbf{7}(1), 311--350 (1974).
\newblock \doi{10.1007/BF01585529}

\bibitem{Gill1981}
Gill, P.E., Murray, W., Wright, M.H.: Practical optimization.
\newblock Academic press (1981)

\bibitem{Goldfeld1966}
Goldfeld, S.M., Quandt, R.E., Trotter, H.F.: {Maximization by Quadratic
  Hill-Climbing}.
\newblock Econometrica \textbf{34}(3), 541--551 (1966)

\bibitem{Golub1996}
Golub, G.H., Van~Loan, C.F.: {Matrix Computations}, third edn.
\newblock The Johns Hopkins University Press, Baltimore, MD, USA (1996)

\bibitem{Higham2016b}
Higham, N., Strabić, N., Šego, V.: {Restoring Definiteness via Shrinking,
  with an Application to Correlation Matrices with a Fixed Block}.
\newblock SIAM Review \textbf{58}(2), 245--263 (2016).
\newblock \doi{10.1137/140996112}.
\newblock \urlprefix\url{https://doi.org/10.1137/140996112}

\bibitem{Higham1988}
Higham, N.J.: Computing a nearest symmetric positive semidefinite matrix.
\newblock Linear Algebra and its Applications \textbf{103}, 103--118 (1988).
\newblock \doi{10.1016/0024-3795(88)90223-6}.
\newblock
  \urlprefix\url{http://www.sciencedirect.com/science/article/pii/0024379588902236}

\bibitem{Higham1989}
Higham, N.J.: {Matrix Nearness Problems and Applications}.
\newblock In: M.J.C. Gover, S.~Barnett (eds.) Applications of Matrix Theory,
  pp. 1--27. Oxford University Press (1989)

\bibitem{Higham2002}
Higham, N.J.: {Accuracy and Stability of Numerical Algorithms}, 2nd edn.
\newblock Society for Industrial and Applied Mathematics, Philadelphia, PA, USA
  (2002).
\newblock \doi{10.1137/1.9780898718027}

\bibitem{Higham2002a}
Higham, N.J.: {Computing the Nearest Correlation Matrix - A Problem from
  Finance}.
\newblock IMA J. Numer. Anal. \textbf{22}(3), 329--343 (2002).
\newblock \doi{10.1093/imanum/22.3.329}

\bibitem{Higham2016}
Higham, N.J., Strabi{\'{c}}, N.: {Anderson Acceleration of the Alternating
  Projections Method for Computing the Nearest Correlation Matrix}.
\newblock Numerical Algorithms \textbf{72}(4), 1021--1042 (2016).
\newblock \doi{10.1007/s11075-015-0078-3}.
\newblock \urlprefix\url{https://doi.org/10.1007/s11075-015-0078-3}

\bibitem{Householder1964}
Householder, A.S.: The theory of matrices in numerical analysis.
\newblock Blaisdell Publishing Company (1964)

\bibitem{Ikeda2016}
Ikeda, Y., Kubokawa, T., Srivastava, M.S.: Comparison of linear shrinkage
  estimators of a large covariance matrix in normal and non-normal
  distributions.
\newblock Computational Statistics {\&} Data Analysis \textbf{95}, 95--108
  (2016).
\newblock \doi{10.1016/j.csda.2015.09.011}.
\newblock
  \urlprefix\url{http://www.sciencedirect.com/science/article/pii/S0167947315002388}

\bibitem{Ikramov2001}
Ikramov, K.D.: On the inertia law for normal matrices.
\newblock In: Doklady Mathematics C/C of Doklady Akademii Nauk, vol.~64, pp.
  141--142 (2001)

\bibitem{Iman1982}
Iman, R., Davenport, J.: An iterative algorithm to produce a positive definite
  correlation matrix from an approximate correlation matrix (with a program
  user's guide).
\newblock Tech. rep., Sandia National Labs., Albuquerque, NM (USA) (1982).
\newblock \doi{10.2172/5152227}.
\newblock \urlprefix\url{https://doi.org/10.2172/5152227}

\bibitem{SciPy-1.3}
Jones, E., Oliphant, T., Peterson, P., et~al.: {SciPy: library for scientific
  computing with Python} (2019).
\newblock \urlprefix\url{http://www.scipy.org}.
\newblock Version 1.3

\bibitem{pytest-4.4.1}
Krekel, H., et~al.: {pytest: helps you write better programs} (2019).
\newblock \urlprefix\url{https://docs.pytest.or}.
\newblock Version 4.4.1

\bibitem{Ledoit2003}
Ledoit, O., Wolf, M.: Improved estimation of the covariance matrix of stock
  returns with an application to portfolio selection.
\newblock Journal of Empirical Finance \textbf{10}(5), 603--621 (2003).
\newblock \doi{10.1016/S0927-5398(03)00007-0}.
\newblock
  \urlprefix\url{http://www.sciencedirect.com/science/article/pii/S0927539803000070}

\bibitem{Ledoit2004}
Ledoit, O., Wolf, M.: A well-conditioned estimator for large-dimensional
  covariance matrices.
\newblock Journal of Multivariate Analysis \textbf{88}(2), 365--411 (2004).
\newblock \doi{10.1016/S0047-259X(03)00096-4}.
\newblock
  \urlprefix\url{http://www.sciencedirect.com/science/article/pii/S0047259X03000964}

\bibitem{Levenberg1944}
Levenberg, K.: {A method for the solution of certain problems in least
  squares}.
\newblock Quarterly of applied mathematics \textbf{2}(2), 164--168 (1944)

\bibitem{Marquardt1963}
Marquardt, D.W.: {An Algorithm for Least-Squares Estimation of Nonlinear
  Parameters}.
\newblock Journal of the Society for Industrial and Applied Mathematics
  \textbf{11}(2), 431--441 (1963).
\newblock \urlprefix\url{http://www.jstor.org/stable/2098941}

\bibitem{More1979}
Mor{\'e}, J.J., Sorensen, D.C.: On the use of directions of negative curvature
  in a modified newton method.
\newblock Mathematical Programming \textbf{16}(1), 1--20 (1979).
\newblock \doi{10.1007/BF01582091}.
\newblock \urlprefix\url{https://doi.org/10.1007/BF01582091}

\bibitem{Nocedal2006}
Nocedal, J., Wright, S.: {Numerical Optimization}, second edn.
\newblock Springer series in operations research and financial engineering.
  Springer, New York (2006)

\bibitem{NumPy-1.17}
Oliphant, T.E., et~al.: {NumPy: N-dimensional array package for Python} (2019).
\newblock \urlprefix\url{http://www.numpy.org}.
\newblock Version 1.17

\bibitem{pip}
PyPA: {The Python Package Installer}.
\newblock \urlprefix\url{https://pip.pypa.io}

\bibitem{Python-3.7}
{Python Software Foundation}: {Python} (2018).
\newblock \urlprefix\url{http://www.python.org}.
\newblock Version 3.7

\bibitem{Qi2010}
Qi, H., Sun, D.: Correlation stress testing for value-at-risk:
  an unconstrained convex optimization approach.
\newblock Computational Optimization and Applications \textbf{45}(2), 427--462
  (2010).
\newblock \doi{10.1007/s10589-008-9231-4}.
\newblock \urlprefix\url{https://doi.org/10.1007/s10589-008-9231-4}

\bibitem{matrix-decomposition-anaconda}
Reimer, J.: {Conda package for matrix-decomposition: a library for decompose
  (factorize) dense and sparse matrices in Python}.
\newblock \urlprefix\url{https://anaconda.org/jore/matrix-decomposition}

\bibitem{matrix-decomposition-git}
Reimer, J.: {GitHub repository for matrix-decomposition: a library for
  decompose (factorize) dense and sparse matrices in Python}.
\newblock https://github.com/jor-/matrix-decomposition.
\newblock \urlprefix\url{https://github.com/jor-/matrix-decomposition}

\bibitem{matrix-decomposition-PyPI}
Reimer, J.: {Python package for matrix-decomposition: a library for decompose
  (factorize) dense and sparse matrices in Python}.
\newblock \urlprefix\url{https://pypi.org/project/matrix-decomposition/}

\bibitem{matrix-decomposition-1.2}
Reimer, J.: {matrix-decomposition: a library for decompose (factorize) dense
  and sparse matrices in Python} (2019).
\newblock \doi{10.5281/zenodo.3558540}.
\newblock \urlprefix\url{https://doi.org/10.5281/zenodo.3558540}.
\newblock Version 1.2

\bibitem{scikit-sparse-0.4.4}
Reimer, J., Grigorievskiy, A., Lee, A., Yuri, Barrett, L., Seljebotn, D.S.,
  Smith, N., Cournapeau, D.: {scikit-sparse: a library for sparse matrix
  manipulation in Python} (2018).
\newblock \urlprefix\url{https://github.com/scikit-sparse/scikit-sparse}.
\newblock Version 0.4.4

\bibitem{Rousseeuw1993}
Rousseeuw, P.J., Molenberghs, G.: Transformation of non positive semidefinite
  correlation matrices.
\newblock Communications in Statistics--Theory and Methods \textbf{22}(4),
  965--984 (1993)

\bibitem{Rudin1976}
Rudin, W.: Principles of Mathematical Analysis.
\newblock International series in pure and applied mathematics. McGraw-Hill
  (1976)

\bibitem{Schaefer2005}
Sch{\"a}fer, J., Strimmer, K.: {A Shrinkage Approach to Large-Scale Covariance
  Matrix Estimation and Implications for Functional Genomics}.
\newblock Statistical Applications in Genetics and Molecular Biology
  \textbf{4}(1), 1--32 (2005)

\bibitem{Schnabel1990}
Schnabel, R., Eskow, E.: {A New Modified Cholesky Factorization}.
\newblock SIAM Journal on Scientific and Statistical Computing \textbf{11}(6),
  1136--1158 (1990).
\newblock \doi{10.1137/0911064}

\bibitem{Schnabel1999}
Schnabel, R., Eskow, E.: {A Revised Modified Cholesky Factorization Algorithm}.
\newblock SIAM Journal on Optimization \textbf{9}(4), 1135--1148 (1999).
\newblock \doi{10.1137/S105262349833266X}

\bibitem{Stein1956}
Stein, C.: {Inadmissibility of the Usual Estimator for the Mean of a
  Multivariate Normal Distribution}.
\newblock In: Proceedings of the Third Berkeley Symposium on Mathematical
  Statistics and Probability, Volume 1: Contributions to the Theory of
  Statistics, pp. 197--206. University of California Press, Berkeley, Calif.
  (1956)

\bibitem{Stewart1980}
Stewart, G.: {The Efficient Generation of Random Orthogonal Matrices with an
  Application to Condition Estimators}.
\newblock SIAM Journal on Numerical Analysis \textbf{17}(3), 403--409 (1980).
\newblock \doi{10.1137/0717034}.
\newblock \urlprefix\url{https://doi.org/10.1137/0717034}

\bibitem{Sylvester1852}
Sylvester, J.J.: A demonstration of the theorem that every homogeneous
  quadratic polynomial is reducible by real orthogonal substitutions to the
  form of a sum of positive and negative squares.
\newblock The London, Edinburgh, and Dublin Philosophical Magazine and Journal
  of Science \textbf{4}(23), 138--142 (1852)

\bibitem{Touloumis2015}
Touloumis, A.: {Nonparametric Stein-type shrinkage covariance matrix estimators
  in high-dimensional settings}.
\newblock Computational Statistics {\&} Data Analysis \textbf{83}, 251--261
  (2015).
\newblock \doi{10.1016/j.csda.2014.10.018}.
\newblock
  \urlprefix\url{http://www.sciencedirect.com/science/article/pii/S0167947314003107}

\bibitem{Virtanen2019}
Virtanen, P., Gommers, R., Oliphant, T.E., Haberland, M., Reddy, T.,
  Cournapeau, D., Burovski, E., Peterson, P., Weckesser, W., Bright, J.,
  van~der Walt, S., Brett, M., Wilson, J., Millman, K.J., Mayorov, N., Nelson,
  A.R.J., Jones, E., Kern, R., Larson, E., Carey, C.J., Polat, I., Feng, Y.,
  Moore, E.W., VanderPlas, J., Laxalde, D., Perktold, J., Cimrman, R.,
  Henriksen, I., Quintero, E.A., Harris, C.R., Archibald, A.M., Ribeiro, A.H.,
  Pedregosa, F., van Mulbregt, P., SciPy: {SciPy 1.0-Fundamental Algorithms for
  Scientific Computing in Python}.
\newblock CoRR \textbf{abs/1907.10121} (2019).
\newblock \urlprefix\url{http://arxiv.org/abs/1907.10121}

\bibitem{Yannakakis1981}
Yannakakis, M.: {Computing the Minimum Fill-In is NP-Complete}.
\newblock SIAM Journal on Algebraic Discrete Methods \textbf{2}(1), 77--79
  (1981).
\newblock \doi{10.1137/0602010}

\end{thebibliography}
}
\end{document}